\documentclass[12pt]{amsart}

\usepackage{graphicx,amssymb,amsmath,xcolor,enumerate}
\usepackage{hyperref}
\hypersetup{
	bookmarks=true,         
	pdffitwindow=false,     
	pdfstartview={FitH},    
	colorlinks=true,      
	citecolor=red,
}
\usepackage{dsfont}
\usepackage{bbm}
\usepackage{enumitem}

\usepackage{float}

\makeatletter
\def\l@subsection{\@tocline{2}{0pt}{2.5pc}{5pc}{}}
\makeatother

\usepackage{geometry}
\DeclareSymbolFont{largesymbol}{OMX}{yhex}{m}{n}
\DeclareMathAccent{\Widehat}{\mathord}{largesymbol}{"62}

\renewcommand{\div}{{\rm div \thinspace }}

\numberwithin{equation}{section}              
\newtheorem{theorem}{Theorem}[section]

\newtheorem{lemma}{Lemma}[section]
\newtheorem{proposition}{Proposition}[section]
\newtheorem*{proposition*}{Proposition}

\newtheorem*{corollary*}{Corollary}
\newtheorem{definition}{Definition}[section]
\newtheorem*{definitions*}{Definitions}

\newtheorem*{acknowledgements*}{Acknowledgements}

\newtheorem*{conjecture*}{\bf Conjecture}

\newtheorem*{example*}{\bf Example}
\theoremstyle{remark}
\newtheorem{remark}{\bf Remark}[section]

\newcommand{\Bv}{{\boldsymbol{v}}}

\begin{document}
	\date{}                                     

\author{Changfeng Gui}
\address{
Department of Mathematics, University of Macau, Taipa, Macau and
Zhuhai UM Science and Technology Research Institute, Hengqin, Guangdong, 519031, China}
\email{changfenggui@um.edu.mo}

	\author{Hao Liu}
	\address{Department of Mathematics, Faculty of Science and Technology, University of Macau, Taipa, Macao}
	\email{haoliu@um.edu.mo}
	\author{Chunjing Xie}
	\address{School of Mathematical Sciences, Institute of Natural Sciences,
		Ministry of Education Key Laboratory of Scientific and Engineering Computing,
		and CMA-Shanghai, Shanghai Jiao Tong University, 800 Dongchuan Road, Shanghai, China}
	\email{cjxie@sjtu.edu.cn}
	
	\title[Forward self-similar solutions to 2D Navier-Stokes equations]{On  the forward self-similar solutions to the two-dimensional Navier-Stokes equations}
	

	\begin{abstract}

		We establish the global existence of forward self-similar solutions to the two-dimensional incompressible Navier-Stokes equations for any divergence-free initial velocity that is homogeneous of degree $-1$ and locally H\"older continuous. This result requires no smallness assumption on the initial data.
		In sharp contrast to the three-dimensional case, where $(-1)$-homogeneous vector fields are locally square-integrable, the major difficulty for the 2D problem is the criticality in the sense that the initial kinetic energy is locally infinite at the origin, and the initial vorticity fails to be locally integrable, so that the classical local energy estimates are not available.
		Our key ideas are to  decompose the solution into a linear part solving the heat equation and a finite-energy perturbation part, and to exploit  a kind of inherent cancellation relation between the linear part and the perturbation part. These, together with suitable choices of multipliers, enable us to control  the interaction terms and to establish the $H^1$-estimates for the perturbation part. 
		Furthermore,  we can get an optimal pointwise estimate via investigating the corresponding Leray equations in  weighted Sobolev spaces.
		This gives the faster decay of the perturbation part at infinity and compactness, which play important roles in proving
		 the existence of global-in-time self-similar solutions.
	\end{abstract}
    
    \maketitle
    
	\section{Introduction and Main Results}

	We consider the two-dimensional Navier-Stokes equations
	\begin{equation}\label{eq:NS}
		\left\{
		\begin{aligned}
			&\frac{\partial  {u}}{\partial t} -\Delta {u} +  {u}\cdot \nabla  {u} +\nabla \pi =0,\\
			&\div   {u} = 0, \\
		\end{aligned}
		\right.  \textrm{ in } \mathbb{R}^2\times (0,\infty)
	\end{equation}
	with initial condition 
	\begin{equation}\label{eq:initial}
		u(x,0)=u_0(x).
	\end{equation}
	Here $ u(x,t): \mathbb{R}^2\times[0,+\infty)\to \mathbb{R}^2$ is the velocity vector field and the scalar $\pi (x,t):\mathbb{R}^2\times[0,+\infty)\to \mathbb{R}$ is the pressure. 
	The  problem \eqref{eq:NS} is invariant under the scaling
	\begin{equation*}
		\begin{aligned}
			&u (x,t) \to u_{\lambda} (x,t) = \lambda u (\lambda x,\lambda^2 t),\\
			&\pi (x,t) \to \pi _{\lambda}(x,t)= \lambda^2 \pi (\lambda x,\lambda^2 t).\\
			& u_0(x) \to u_{0,\lambda}(x) =\lambda u_0(\lambda x).
		\end{aligned}
	\end{equation*}
	In particular, the solution $u$  is called  self-similar or  scale-invariant  if
	$u_{\lambda}= u$  and $\pi _{\lambda}=\pi $ for any $\lambda > 0$.
	Similarly, we say that the  initial condition $u_0$ is self-similar or scale-invariant,
	if $u_0$ satisfies $\lambda u_0(\lambda x)= u_0(x)$ any for $\lambda > 0$.
	
	The study of self-similar solutions is motivated by their fundamental roles in understanding the well-posedness, especially the structure of singularities, the asymptotic behavior, and the uniqueness or non-uniqueness of the Navier-Stokes equations. In the context of the regularity problem, (backward) self-similar solutions often serve as natural candidates for blow-up profiles, and have been excluded in the 3D case (\cite{Sverak96, Tsai98}). On the other hand,
    forward self-similar solutions, which arise from scale-invariant initial data, 
    provide a critical setting to test the robustness of the well-posedness theory. Recently, a lot of important progress has been made on the non-uniqueness of Leray-Hopf solutions through the study of the forward 
    self-similar solutions
     \cite{Jia15, Albritton22annals, Guillod23, Thomas25}.

The existence of solutions  in a scale-invariant space has been studied extensively, including $L^{n}$, $L^{n,\infty}$, Besov spaces and $BMO^{-1}$ in $n$ dimensions (usually $n=2,3$), typically under \emph{smallness assumptions}, see \cite{Barraza96, Cannone96, Giga89, Koch01}. Especially, the small data existence of forward self-similar solutions follows from the uniqueness results in these studies.
	One may also refer to \cite[Chapter 8]{Tsai18} and the reference therein. 

In the three-dimensional case,	
	recently, the seminal work of Jia and {\v{S}}ver{\'{a}}k \cite{Jia14} established the global existence of smooth self-similar solutions for the 3D Navier-Stokes equations with any Hölder continuous self-similar initial data on $\mathbb{R}^{3} \setminus \{0\}$. The main techniques of their approach are to derive  the local-in-space
	H\"older estimates based on the $L_{loc}^{2}$-energy estimates established
	by Lemari\'{e}-Rieusset for the local Leray solutions in \cite{Lemarie02}.  Later in 
	\cite{Korobkov2016}, the self-similar solutions  were constructed
	based on the a priori $H^1$-estimate obtained for  the equivalent Leray equations
	by Leray’s method of contradiction, which is a kind of blow-up argument. In
	\cite{Tsai17, Bradshaw17}, Bradshaw and Tsai
	gave a rather direct method to establish the $H^1$-estimates for both discretely self-similar solutions and rotated discretely self-similar solutions, and then constructed   solutions  via the  Galerkin approximation.

However, the general forward self-similar solutions for two-dimensional Navier-Stokes have neither locally finite energy nor integrable vorticity. This makes the analysis much more challenging.
In the two-dimensional case,			note that the (-1)-homogeneous divergence-free velocity field  $u_0(x)$ has constant circulation on any circle, and be decomposed  as 
	\begin{equation}\label{IC1}
		u_0(x) = \frac{\alpha}{2\pi} \frac{x^\perp}{|x|^2}   +\tilde{u}_0,
	\end{equation}
	where $\tilde{u}_0$ has zero circulation.
In the case $\tilde{u}_0=0$, an explicit class of self-similar solutions is known as Oseen vortices,  whose velocity and vorticity are given by 
	\begin{equation*}
		u(x,t) = \frac{\alpha}{2\pi}  \frac{x^\perp}{|x|^2}(1-e^{-\frac{|x|^2}{4t}}) \,\,\text{ and }\,\, ~\omega (x,t) = \frac{\alpha}{4\pi t} e^{-\frac{|x|^2}{4t}},
	\end{equation*}
    respectively.
	In fact, the  
	Oseen vortices characterize the long-time asymptotic behavior of general solutions to \eqref{eq:NS} with integrable initial vorticity \cite{Gallay05, Giga88, Carpio94}. Therefore, the forward self-similar solutions also play crucial roles in understanding the long-time behavior of general solutions of  \eqref{eq:NS}. 
	However, there is no known existence result of self-similar solutions with the intial data of the form \eqref{IC1} with large $\tilde{u}_0$, even for the special case $\alpha=0$.

    \subsection{Main results}
	Our main goal is to construct a global forward self-similar solution for any self-similar initial velocity field, generally with non-integrable vorticity. Before stating our result, we give the definition of the solution to \eqref{eq:NS}, which is essentially the same as \cite[Definition 8.18]{Tsai18}. 
	\begin{definition}\label{def:energy persol}
		A vector field $u$ defined on $\mathbb{R}^2\times[0,\infty)$ is called an energy perturbed solution with divergence-free initial data $u_0\in L^{2,\infty}(\mathbb{R}^2) $ 
		if it satisfies  \eqref{eq:NS} in the sense of distributions, and \begin{itemize}
			\item[(i)]  $u- e^{t\Delta}u_0 \in L^{\infty}(0,T;L^2(\mathbb{R}^2))\cap L^{2}(0,T;H^1(\mathbb{R}^2))$ for any $T>0$,
			\item[(ii)] $\lim_{t\to 0} \|u(t)- e^{t\Delta}u_0 \|_{L^2(\mathbb{R}^2)} = 0$,
		\end{itemize}
		where \begin{equation*}
			(e^{t\Delta}u_0) (x) = \int_{\mathbb{R}^2} \frac{1}{4\pi t}e^{-\frac{|x-z|^2}{4 t}}u_0(z)dz.
		\end{equation*}
		is the solution to the heat equation with the same initial data $u_0$.
	\end{definition}
	
	The main results in this paper are as follows.
	\begin{theorem}\label{thm:main}
		Assume $u_0(x)\in C^{0,\beta}_{loc}$ for some $0 < \beta\leq 1$ is a divergence-free, self-similar  initial data  in $\mathbb{R}^2\setminus\{0\}$, and that
		$  \int_{ S^1} u_0\cdot n d\sigma = 0 $.
        Here, $S^1$ is the unit circle on $\mathbb{R}^2$ and  $C^{0,\beta}_{loc}$ is the standard local H\"older space, when $\beta=1$, $C^{0,1}_{loc}$ is the local Lipshitz space. 
        Suppose that
		\begin{equation*}
			A=\|u_0\|_{C^{0,\beta}(S^1)}<+\infty. 
		\end{equation*}
		Then the following statements hold.
        \begin{itemize}
            \item[(i)]  
		There exists at least one global  smooth self-similar solution $u (x,t)\in C^{\infty}(\mathbb{R}^2\times (0,\infty)) $ to the Cauchy problem \eqref{eq:NS}-\eqref{eq:initial}, 
  which is indeed an energy perturbed solution in the sense of Definition \ref{def:energy persol}. 
        \item[(ii)]  
		The following pointwise estimates hold
		\begin{equation}\label{eq:pointwisedecay}
			|u(x,t)|\leq \frac{C(A,\beta)}{|x|+\sqrt{t}},\quad  ~|u(x,t)- e^{t\Delta}u_0|\le C(A,\beta)\frac{\sqrt{t}}{|x|^2+t},
		\end{equation}
		and 
		\begin{equation*}
			|\nabla u(x,t)- \nabla e^{t\Delta}u_0|\le C(A,\beta)\frac{\sqrt{t}^{\beta}}{(|x|+\sqrt{t})^{2+\beta}}.
		\end{equation*}
        \item[(iii)] 
		If, in addition, the initial data $u_0$ belongs to $C^2(S^1)$,  then we have the following optimal decay rate estimates
		\begin{equation}\label{eq:optimaldecay}
			| u(x,t)- e^{t\Delta}u_0| \le C(\|u_0\|_{C^2(S^1)})\frac{t}{(|x|+\sqrt{t})^3}\ln\left(1+\frac{|x|}{\sqrt{t}}\right).
		\end{equation}
         \end{itemize}
	\end{theorem}
	
	There are a few remarks in order. 
    \begin{remark}\label{rem:deltaatorigin}
It is interesting to note that   any divergence-free  self-similar vector field $u_0\in L^{\infty}_{loc}(\mathbb{R}^3\setminus\{0\})$  automatically satisfies $\div  u_0=0$ across the origin.   However, this property does not hold for a two-dimensional vector field. For example,
\begin{equation*}
    u_0(x) = \frac{1}{2\pi}\frac{x}{|x|^2}
\end{equation*}
is divergence-free on $\mathbb{R}^2\setminus\{0\}$ but $\div u_0=\delta$ in the sense of distributions on $\mathbb{R}^2$.
    Since   the Navier-Stokes equations ask the velocity field to be divergence-free for $t>0$, it is natural to require the initial data $u_0$ to be 
    divergence-free in $\mathbb{R}^2$ at least in the sense of distributions, i.e., $\int_{S^1} u_0\cdot n d\sigma =0$.

	\end{remark}

    \begin{remark}
        Together with the techniques developed in \cite{Tsai2014, Xue15, Chae18AIHP, Bradshaw19apde, Bradshaw18arma, Albritton19arma}, where large self-similar solutions were constructed  with initial data that  lie in $L^2_{loc}(\mathbb{R}^3)$ and Besov space, etc,
        we may reduce the regularity assumptions in Theorem \ref{thm:main} to construct two-dimensional forward self-similar solutions with general initial data.
    \end{remark}

\begin{remark}
The leading term $e^{t\Delta}u_0$ satisfies 
    \begin{equation*}
        \left|e^{t\Delta}u_0(x)\right| \leq \frac{C(A)}{|x|+\sqrt{t}}.
    \end{equation*}
    The details are referred to Lemma \ref{lem:lineardecay}.
\end{remark}

    \begin{remark}
		The decay rate \eqref{eq:optimaldecay} is optimal and cannot be improved in general, even if $u_0$ has higher regularity. One can see  the precise  asymptotic expansions    in \cite{Brandolese09} when the initial data is small.  The remainder term $\frac{t}{(|x|+\sqrt{t})^3}\ln\left(1+\frac{|x|}{\sqrt{t}}\right)$ can not vanish unless the orthogonality relations 
		\begin{equation*}
			\int_{S^1} u_{0,1}^2=	\int_{S^1} u_{0,2}^2 ~\textrm{ and }~\int_{S^1} u_{0,1} u_{0,2}=0
		\end{equation*}
		are satisfied.
	\end{remark}

    \begin{remark}
 In recent years,   the  self-similar solutions of the steady Navier-Stokes equations  were also studied extensively, both on the rigidity part and the existence part. For example,  nontrivial self-similar solutions were obtained in  \cite{Hamel17, Jeffery01041915, Landau44}:
Hamel solutions when the dimension $n=2$ and Landau solutions when $n=3$. The rigidity of Hamel solutions and Landau solutions in the class of  self-similar solutions was established  in \cite{Tian98, Sverak11}.
    The importance of these self-similar solutions  in understanding singularity and far-field behavior of solutions to the steady Navier-Stokes equations was illustrated in \cite{MiuraTsai12, KorolevSverak11} and references therein. There is also a great interest in studying the high-dimensional steady Navier-Stokes equations, and 
    it is known that for  dimensions $n\geq 4$, all self-similar solutions are trivial (see \cite{Tsai98b}, \cite[Corollary in Section 2]{Sverak11}).
This rigidity result has been extended to all solutions with scale-invariant bounds $C/|x|$ without any smallness assumption on $C$ in \cite{bang2023rigidity}.
The existence of self-similar solutions to steady Navier-Stokes equations with small scale-invariant forces has been obtained in \cite{Kozono95, Phan13, Kaneko19, Tsurumi19} for  $n\geq 3$. Recently, the existence of self-similar solutions with arbitrarily large scale-invariant forces has been proved in \cite{Hao25} up to dimension 16.
\end{remark}

\begin{remark}
For the constructions of self-similar solutions for the fractional Navier-Stokes equations and other related models, such as MHD and SQG equation, one can refer to \cite{Lai19,  Lai24, Zhang19jmp, Yang26, Albritton22}. It is also possible to study forward self-similar solutions for the associated 2D system with the aid of the methods developed in this paper.
\end{remark}

    \begin{remark}
        We recently got to know that Dallas Albritton, Julien Guillod, Mikhail Korobkov, and Xiao Ren (\cite{AGKR}) obtained a result similar to Theorem \ref{thm:main}, where  $u(x,t)- e^{t\Delta}u_0$ at $t=1$ admits spatial decay $|x|^{-1-\beta}$ if the initial data $ u_0$ belongs to $ C^{0,\beta}(S^1)$. After both papers appeared on ArXiv,  we noticed that the methods are quite different. We use straightforward energy estimates and do not use special quantities like the Bernoulli pressure.
    \end{remark}

	\subsection{Main difficulties and the key ideas}
	The main difficulty in proving Theorem \ref{thm:main}, compared to the 3D case, lies in the non-integrability of the initial data, even locally, at the level of both velocity and vorticity in $\mathbb{R}^2$. In 2D, a self-similar initial velocity  $|u_0(x)| \sim |x|^{-1}$ does not belong to $L_{loc}^{2}(\mathbb{R}^{2})$, 
	hence, the local Leray solutions developed  in \cite{Lemarie02} are not applicable.
	On the other hand,  the initial vorticity does not belong to $L_{loc}^{1}(\mathbb{R}^{2})$ and even cannot be viewed as a finite measure as $|\textrm{curl} ~u_0(x)| \sim |x|^{-2}$. Consequently, we lack a fundamental quantity to initiate the estimates at both the velocity and the vorticity level.
	
	To overcome this lack of initial integrability, we adopt a decomposition strategy. We postulate that the singular behavior of the solution is dominated by the linear evolution. Let $ e^{t\Delta}u_0$ be the caloric lift of the initial data, i.e., the solution to the heat equation starting from the same initial data $u_0$. While $ e^{t\Delta}u_0$ retains the non-$L^2$ character of the initial data, it is explicitly computable and smooth for $t > 0$. The unknown remainder term $u_{\text{re}} = u -  e^{t\Delta}u_0$ is then expected to possess better  decay properties. 
     The core of our proof lies in establishing global-in-time energy estimates for $u_{\text{re}}$, depending only on initial data.

 Due to the self-similarity, we only need to focus on the energy estimates at $t=1$.
   To do so, we use the self-similar variables and substitute this decomposition into the Navier-Stokes equations, thereby deriving the perturbed Leray equations. 
More precisely,  we let
\begin{equation*}
		u(x,t)=\frac{1}{\sqrt{t}}v\left(\frac{x}{\sqrt{t}}\right)=\frac{1}{\sqrt{t}}v(y).
	\end{equation*}
    and
    \begin{equation*}
		v_0(y)= e^{\Delta}u_0(y), ~ v_{\text{re}}=v-v_0.
	\end{equation*}
We can derive that (see Section \ref{sec:Lerayeq})
    \begin{equation*}
		\left\{
		\begin{aligned}
			&-\Delta v_{\text{re}} -\frac{1}{2}v_{\text{re}} -\frac{1}{2}y\cdot \nabla v_{\text{re}} +  v_0\cdot \nabla v_{\text{re}} + v_{\text{re}}\cdot \nabla  v_0 + v_{\text{re}}\cdot \nabla v_{\text{re}} +\nabla q=- v_0\cdot\nabla  v_0, \\
			&  \div v_{\text{re}}=0,
		\end{aligned}
		\right. \textrm{ in }\mathbb{R}^2.
	\end{equation*}   
  The energy estimates of $v_{\text{re}}$ cannot be achieved directly if one tests the equation of $v_{\textrm{re}}$ above by $v_{\text{re}}$. This is a key difference compared with the 
three-dimensional case,  where there is an additional term $\frac{1}{4}\int_{\mathbb{R}^2} |v_{\text{re}}|^2$ 
	 helps to close the energy estimates; see \cite{Tsai17, Bradshaw17}.
The difficulty at this stage is how to derive some initial, although maybe insufficient, estimates for $v_{\text{re}}$.

	Our first key observation  is that one can  control the Dirichlet energy $\|\nabla v_{\text{re}}\|_{L^2(\mathbb{R}^2)}^2$ of $v_{\text{re}}$  by using some inherent \emph{cancellation relation} between $ v_0$ and $v_{\text{re}}$. %
    Indeed, this relation leads to a
 crucial identity (see also \eqref{eq:crucialidentity} below), specific to 2D, which is
    \begin{equation*}
		\int_{\mathbb{R}^2} |\nabla v_{\text{re}}|^2 dy + 2\int_{\mathbb{R}^2} \nabla  v_0\cdot \nabla v_{\text{re}} dy =0.
	\end{equation*}
      And one sees that   $\|\nabla v_{\text{re}}\|_{L^2(\mathbb{R}^2)}$ can be controlled by $\|\nabla v_0\|_{L^2(\mathbb{R}^2)}$, which is finite by using explicit estimates for  $ e^{\Delta}u_0$. This Dirichlet energy estimate $\|\nabla v_{\text{re}}\|_{L^2(\mathbb{R}^2)}$ for $ v_{\text{re}}$ is the starting point of our subsequent analysis.

      Our next  important observation is that the convection term $-\frac{1}{2}y\cdot \nabla v_{\text{re}}$ can be more effectively used if we choose the multiplier $|v_{\text{re}}|^{p-2}v_{\text{re}}$, 
where $1<p<2$. 
Testing the equation of $v_{\textrm{re}}$ above, the linear  term $-\frac{1}{2}v_{\text{re}} -\frac{1}{2}y\cdot \nabla v_{\text{re}}$
gives a positive quantity
\begin{equation*}
 \int_{\mathbb{R}^2} (-\frac{1}{2}v_{\text{re}} -\frac{1}{2}y\cdot \nabla v_{\text{re}})\cdot |v_{\text{re}}|^{p-2}v_{\text{re}} =
 \int_{\mathbb{R}^2}\left(\frac{1}{p}-\frac{1}{2}\right) |v_{\text{re}}|^{p}dy. 
\end{equation*}
This quantity, together with the estimate of $\|\nabla v_{\text{re}}\|_{L^2(\mathbb{R}^2)}$ is
 essential  for us to close the $L^p$-estimate of $v_{\text{re}}$ for any $1<p<2$. This also implies that $v_{\textrm{re}}$  decays like  $\frac{1}{1+|y|^2}$  in an average sense.  The application of  interpolation inequalities  leads to the $L^2$-estimate, and hence,  $H^1$-estimate of $v_{\text{re}}$. One can refer to Section \ref{sec:outline} for a more detailed outline of the proof.

	Once we get the $H^1$-estimate for $v_{\text{re}}$, we proceed to establish the weighted energy estimates and to derive  pointwise estimates for $v_{\text{re}}$ inspired by the methods developed in  \cite{Lai19}. Specifically, we choose the test function $|y|^2 v_{\text{re}}$, and the linear  term $-\frac{1}{2}v_{\text{re}} -\frac{1}{2}y\cdot \nabla v_{\text{re}}$
    contributes a crucial positive term  
    \begin{equation*}
				 \int_{\mathbb{R}^2} \left( -\frac{1}{2} v_{\text{re}} - \frac{1}{2} y \cdot \nabla v_{\text{re}} \right) \cdot (|y|^2 v_{\text{re}}) \, dy =\frac{1}{2} \int_{\mathbb{R}^2} |y|^2 |v_{\text{re}}|^2 \, dy.
			\end{equation*}
            This integral, due to the convection $- \frac{1}{2} y \cdot  \nabla v_{\text{re}}$, together with $H^1$-estimate of $v_{\text{re}}$  leads to 
            $H^1$-estimate of $|y| v_{\text{re}}$. After bootstrapping to $H^2$-estimate of $|y| v_{\text{re}}$, 
    we have the following pointwise estimate of $v_{\text{re}}$ by Sobolev embedding, 
	\begin{equation*}
		|v_{\text{re}}(y)| \leq \frac{C}{	1 + |y| } \textrm{ for all } y \in \mathbb{R}^2.
	\end{equation*}
	With this decay estimate in hand, we can further improve the decay estimates by  using the linear theory of Stokes and the heat equation to have that
	\begin{equation*}
		|v_{\text{re}} (y)| \leq \frac{C}{	1 + |y|^2 } \textrm{ for all } y \in \mathbb{R}^2.
	\end{equation*}
	This faster decay compared with the linear part $e^{\Delta}u_0$, together with standard interior regularity, gives compactness, so that we  can employ the Leray-Schauder degree theory to find a solution $v_{\textrm{re}}$, hence, a self-similar solution to the Navier-Stokes equation \eqref{eq:NS}. On the other hand, if $u_0\in C^2_{loc}(\mathbb{R}^2)$, we can utilize the computations in \cite{Brandolese09} to show the optimal decay rate \eqref{eq:optimaldecay}.

	The rest of the paper is organized as follows. We prove the energy estimates for the remainder term $u_{\text{re}} (\cdot,1)= u(\cdot,1) -  e^{\Delta}u_0$ in Section \ref{sec:energy estimates}. Next, we establish the weighted energy estimates and derive pointwise decay estimates for $u_{\text{re}} (\cdot,1)$  in Section \ref{sec:pointwise}. Finally,  the main theorem is proved in Section \ref{sec:mainthm} by using the Leray-Schauder degree theory.

	\section{The  Leray equations and the Energy estimates}\label{sec:energy estimates}
	The primary objective of this section is to establish the $H^1(\mathbb{R}^2)$-energy estimates for the remainder term defined by
	\begin{equation*}
		u_{\text{re}}(x,1)= \left.u - e^{t\Delta}u_0\right.|_{t=1}= u(x,1)- e^{\Delta}u_0(x).
	\end{equation*}  
	Furthermore, we demonstrate that $	u_{\text{re}}(x,1)$ is $L^{p}$-integrable for any $p \in (1, +\infty)$. This implies that $u_{\textrm{re}}$  decays like $\frac{1}{1+|x|^2}$ in an average sense.
	The local regularity of $u_{\text{re}}(x,1)$ and $e^{\Delta}u_0$ is not a big issue, indeed they are smooth due to the parabolic smoothing effect, once we have some initial local integrability. However, they decay slowly at spatial infinity, and the key point is to get the global integrability to initiate the subsequent analysis. Before stating the main result in this section, we introduce self-similar variables and the Leray equations.

	\subsection{Self-similar variables and the Leray equations}\label{sec:Lerayeq}
	Our analysis will be conducted in self-similar variables; we let
	\begin{equation}\label{eq:changeofvaribales}
		y=\frac{x}{\sqrt{t}} ~\textrm{and}~ s=\log (t).
	\end{equation}
	We then define   $v(y,s)$  by
	\begin{equation}\label{eq:changeofvareq}
		u(x,t)=\frac{1}{\sqrt{t}}v(y,s) = \frac{1}{\sqrt{t}}v\left(\frac{x}{\sqrt{t}},\log t\right).
	\end{equation}
	It follows that if $u$ solves \eqref{eq:NS}, then $v$ satisfies the time-dependent \emph{Leray equations}:
	\begin{equation}\label{eq:Leray}
		\left\{
		\begin{aligned}
			&\frac{\partial v}{\partial s}-\Delta v -\frac{1}{2}v -\frac{1}{2}y\cdot \nabla v+ v\cdot \nabla v +\nabla q=0,\\
			&\textrm{ div } v=0,
		\end{aligned}
		\right. \textrm{ in } \mathbb{R}^2\times (-\infty,\infty).
	\end{equation}
	The self-similarity of $u$ implies that $v$ is independent of $s$, and in this case
	$v(y)=u(y,1)$. We have the following stationary \emph{Leray equations} for $v$
	\begin{equation}\label{eq:NSinself-similar}
		\left\{
		\begin{aligned}
			&-\Delta v -\frac{1}{2}v -\frac{1}{2}y\cdot \nabla v+ v\cdot \nabla v +\nabla q=0,\\
			&\textrm{ div } v=0,
		\end{aligned}
		\right. \textrm{ in } \mathbb{R}^2.
	\end{equation}
	As is expected $ v\sim \frac{C}{1+|y|}$, hence, $ v\notin L^2(\mathbb{R}^2)$. 
	To circumvent this, we subtract the linear part from $v$ using the following decomposition.
	Let $ e^{t\Delta}u_0$ be the solution to the heat equation with initial data $u_0$, that  is 
	\begin{equation*}
		e^{t\Delta}u_0(x) =\int_{\mathbb{R}^2} \frac{1}{4\pi t}e^{-\frac{|x-z|^2}{4 t}}u_0(z)dz.
	\end{equation*}
	Direct calculation shows that $e^{t\Delta}u_0$ inherits the self-similarity of $u_0$, i.e.,
	\begin{equation*}
		\begin{aligned}
			\lambda  (e^{\lambda^2 t\Delta}u_0) (\lambda x) & = \lambda\int_{\mathbb{R}^2} \frac{1}{4\pi \lambda^2 t}e^{-\frac{|\lambda x- z|^2}{4 \lambda^2 t}}u_0(z)dz\\
			&=\lambda\int_{\mathbb{R}^2} \frac{1}{4\pi \lambda^2 t}e^{-\frac{|\lambda x- \lambda z|^2}{4 \lambda^2 t}}u_0(\lambda z)\lambda^{2} dz\\
			&=\int_{\mathbb{R}^2} \frac{1}{4\pi  t}e^{-\frac{| x-  z|^2}{4 t}}u_0( z) dz= e^{t\Delta}u_0(x),
		\end{aligned}
	\end{equation*}
	where in the last line we have used that $u_0$ is self-similar.
	We define $ v_0(y)$ by
	\begin{equation*}
		\frac{1}{\sqrt{t}}v_0(y)=  (e^{t\Delta}u_0)(x).
	\end{equation*}
Obviously we have
	\begin{equation}\label{eq:defofv1}
		v_0(y) = (e^{\Delta}u_0)\left(\frac{x}{\sqrt{t}}\right) =  (e^{\Delta}u_0)(y).
	\end{equation}
	Let $	v_{\text{re}}$ be the remainder part of $v$ when subtracting the linear part $v_0$, that is
	\begin{equation}\label{eq:deforre}
		v_{\text{re}}(y)=v(y)- v_0(y) = u_{\text{re}}(y,1).
	\end{equation}
	Our primary goal is to get a priori estimates for $	v_{\text{re}}(y)$.
	It is easy to see that $v_0$ satisfies
	\begin{equation}\label{eq:equforv1}
		-\Delta  v_0 -\frac{1}{2} v_0 -\frac{1}{2}y\cdot \nabla  v_0 = 0, ~\div v_0=0.
	\end{equation}
	Hence, $	v_{\text{re}}$ satisfies
	\begin{equation}\label{eq:difference}
		\left\{
		\begin{aligned}
			&-\Delta v_{\text{re}} -\frac{1}{2}v_{\text{re}} -\frac{1}{2}y\cdot \nabla v_{\text{re}}+  v_0\cdot \nabla v_{\text{re}} + v_{\text{re}}\cdot \nabla  v_0 + v_{\text{re}}\cdot \nabla v_{\text{re}} +\nabla q=- v_0\cdot\nabla  v_0, \\
			&  \div v_{\text{re}}=0,
		\end{aligned} 
		\right. \textrm{ in } \mathbb{R}^2.
	\end{equation}

    We define  the following function space for later use.
			\begin{equation}\label{eq:defiofX}
				X=\{ \Bv:\ \div \Bv =0, \ \Bv\in C^1(\mathbb{R}^2), ~ \|\Bv\|_{X}<+\infty\},
			\end{equation}
			where the norm $\|\cdot\|_X$ is 
			\begin{equation}\label{eq:defofnorm}
				\|\Bv\|_{X}= \|(1+|x|^2)\Bv\|_{C(\mathbb{R}^2)} +\|(1+|x|^2)\nabla \Bv\|_{C(\mathbb{R}^2)}.
			\end{equation}
     Indeed, this space $X$ will be the space where we use the Leray-Schauder fixed-point theorem to seek a solution $v_{\text{re}}$ to \eqref{eq:difference}. We can rather assume that $v_{\text{re}}\in X$ to facilitate later computations; the key point is to obtain enough a priori estimates of $v_{\text{re}}$ in $X$ in terms of the initial data $u_0$ only.
     
	The main objective of this section is to prove the following a priori estimates for $v_{\text{re}}$.
	\begin{theorem}\label{thm:energyestkey}
		For $0<\beta\leq 1$, given the divergence-free, self-similar initial data $u_0(x)\in C^{0,\beta}_{loc}$, and
		\begin{equation*}
			A=\|u_0\|_{C^{0,\beta}(S^1)}<+\infty.
		\end{equation*}
		Assume $u$  is a self-similar solution to \eqref{eq:NS} in the sense of Definition \ref{def:energy persol},
		and that 
        $v_{\textrm{re}}$ $ \in X$ is a corresponding solution to
		\eqref{eq:difference} in which $v_0=e^{\Delta}u_0$. Then $v_{\text{re}}$ is smooth and
		there exists a constant $C=C(A,\beta)$ depending only on $A$ and $\beta$, such that
		\begin{equation}\label{eq:energyest}
			\int_{\mathbb{R}^2} |v_{\text{re}}|^2 + |\nabla v_{\text{re}}|^2 dx \leq C(A,\beta).
		\end{equation}
		Moreover,  we have that for any positive constant $1<p\leq \infty$, there is a constant  $C=C(A,\beta,p)$ depending  on $A$, $p$, and $\beta$ that 
		\begin{equation*}
			\|v_{\text{re}}\|_{L^{p}(\mathbb{R}^2)} \leq C(A,\beta,p).
		\end{equation*}
		Furthermore, we have the following estimates for the pressure $q$ in \eqref{eq:difference}. For any positive constant $1<p< \infty$, we have
		\begin{equation*}
			\|q\|_{L^{p}(\mathbb{R}^2)} \leq C(A,\beta,p). 
		\end{equation*}
		and
		also  
		\begin{equation*}
			\|\nabla q\|_{L^{p}(\mathbb{R}^2)} \leq C(A,\beta,p). 
		\end{equation*}
	\end{theorem}

	\subsection{Main difficulties, key ideas and the outline  for the energy estimates}\label{sec:outline}
	In this section, we give the main ideas for proving the a priori energy estimates  in Theorem  \ref{thm:energyestkey}.
	
	Formally multiplying \eqref{eq:difference} by $v_{\text{re}}$ and integrating by parts, we arrive at the energy identity
	\begin{equation}\label{eq:energyesti}
		\int_{\mathbb{R}^2} |\nabla v_{\text{re}}|^2 + ( v_{\text{re}}\cdot \nabla  v_0)\cdot v_{\text{re}} = 	-\int_{\mathbb{R}^2} ( v_0\cdot\nabla  v_0)\cdot  	v_{\text{re}}.
	\end{equation}
	The main difficult lies in controlling the term $\int_{\mathbb{R}^2}  ( v_{\text{re}}\cdot \nabla  v_0)\cdot v_{\text{re}} $. 
	If one uses integration by parts, we have the following estimates
	\begin{equation*}
		\int_{\mathbb{R}^2}  ( v_{\text{re}}\cdot \nabla  v_0)\cdot v_{\text{re}}  = - \int_{\mathbb{R}^2}  ( v_{\text{re}}\cdot \nabla  v_{\text{re}})\cdot  v_0 \leq \| v_0\|_{L^\infty}\int_{\mathbb{R}^2}  ( |v_{\text{re}}|^2+ | \nabla v_{\text{re}}|^2),
	\end{equation*}
	yet $\|v_{\text{re}}\|_{L^2(\mathbb{R}^2)}$ cannot be controlled at this stage. This situation differs from the three-dimensional case,  where there is an additional term $\frac{1}{4}\int_{\mathbb{R}^2} |v_{\text{re}}|^2$ that appears in the left-hand side of \eqref{eq:energyesti} and
	help to close the energy estimates, provided $\|v_0\|_{L^\infty}$ is small via a suitable cut-off; see also \cite{Tsai17, Bradshaw17}.
	
	Our key insight is that an inherent structural relationship exists between $v_0$ and $v_{\text{re}}$, which can be more effectively utilized than a simple size estimate of $\int (v_{\text{re}} \cdot \nabla v_0) \cdot v_{\text{re}}$. 
	Our way to exploit this inherent relation is as follows.
	We	multiply \eqref{eq:difference} by $( v_0+v_{\text{re}})$ and multiply \eqref{eq:equforv1} by $v_{\text{re}}$. Summing the resulting equations and integrating over $\mathbb{R}^2$, we obtain, after integration by parts,  a very crucial identity
	\begin{equation}\label{eq:crucialidentity}
		\int_{\mathbb{R}^2} |\nabla v_{\text{re}}|^2 dy + 2\int_{\mathbb{R}^2} \nabla  v_0\cdot \nabla v_{\text{re}} dy =0.
	\end{equation}
	This, together with the H\"older inequality, leads to 
	\begin{equation}\label{eq:dirichletener}
		\int_{\mathbb{R}^2} |\nabla v_{\text{re}}|^2 dy \leq  4\int_{\mathbb{R}^2}|\nabla  v_0|^2 dy\leq C.
	\end{equation}
	as $\nabla  v_0\sim \frac{C}{1+|y|^{1+\beta}}\in L^2(\mathbb{R}^2)$.
	This provides the initial estimate for the Dirichlet energy of $v_{\text{re}}$.
    
	To obtain the $L^2$ bound for $v_{\text{re}}$, we return to  \eqref{eq:difference} and  try to utilize the good term $-\frac{1}{2}y\cdot \nabla v_{\text{re}}$ as much as possible. The main idea is to use the multiplier $|v_{\text{re}}|^{p-2}v_{\text{re}}$ for some $1<p<2$.
	If we formally test \eqref{eq:difference} with $|v_{\text{re}}|^{p-2}v_{\text{re}}$, we have 
	\begin{equation*}
		\begin{aligned}
			&\int_{\mathbb{R}^2} (p-1)|\nabla v_{\text{re}}|^2 |v_{\text{re}}|^{p-2}  + \left(\frac{1}{p}-\frac{1}{2}\right) |v_{\text{re}}|^{p} \\
			\leq& - \int_{\mathbb{R}^2}  (v_{\text{re}}\cdot \nabla  v_0)\cdot |v_{\text{re}}|^{p-2}v_{\text{re}} + \nabla q\cdot |v_{\text{re}}|^{p-2}v_{\text{re}} + ( v_0\cdot \nabla  v_0)\cdot |v_{\text{re}}|^{p-2} v_{\text{re}}.
		\end{aligned}
	\end{equation*}
	The left hand side contains an additional term $\int_{\mathbb{R}^2} \left(\frac{1}{p}-\frac{1}{2}\right) |v_{\text{re}}|^{p} $, which 
	enables us to close the estimates for
	$\|v_{\text{re}}\|_{L^{p}}$ for any $1<p<2$ eventually. 
	The key issue is to control $\int_{\mathbb{R}^2}  (v_{\text{re}}\cdot \nabla  v_0)\cdot |v_{\text{re}}|^{p-2}v_{\text{re}} $  and $\int_{\mathbb{R}^2}  \nabla q\cdot |v_{\text{re}}|^{p-2}v_{\text{re}}$  as the last term on the right hand side is more manageable.
	To achieve this, we introduce a modified profile of $v_0$ to ensure it is small near the origin. Finally, once the estimate for $\|v_{\text{re}}\|_{L^{p}}$  is established, the Gagliardo–Nirenberg interpolation inequality combined with \eqref{eq:dirichletener} yields the desired $L^2$-boundedness of $v_{\text{re}}$. The higher integrability estimate of $\|v_{\text{re}}\|_{L^{p}}$ when $p>2$ will follow from regularity estimates for the Stokes equation and the bootstrapping argument.

	\subsection{Proof of the energy estimates}
	As preliminaries,
	we begin with the following decay estimates for the solution to the heat equation.  It is more or less standard, and we leave the proof to Appendix \ref{app:lineardecay}.
	\begin{lemma}\label{lem:lineardecay}
		Let $u_0$ be divergence-free on $\mathbb{R}^2$ and self-similar, and let $0\leq \beta\leq 1$, denote
		\begin{equation*}
			A= \|u_0\|_{C^{0, \beta}(S^1)} 		<+\infty.
		\end{equation*}
		Then  $ v_0 = e^{\Delta}u_0(y)$ defined in \eqref{eq:defofv1} is smooth and divergence-free.
		Furthermore, it holds that 
		\begin{equation}\label{eq:scale-invbou-2}
			|  v_0 (y)|\leq \frac{C(A)}{1+|y|}, ~
			|\nabla^k  v_0 (y)|\leq \frac{C(A,\beta, k)}{1+|y|^{1+\beta}}, ~ k\in \mathbb{N}, k\geq 1.
		\end{equation}
	\end{lemma}

		We now proceed to establish energy estimates for $v_{\text{re}} = v - v_0$. As discussed in Section \ref{sec:outline}, we need to construct a modified profile of $v_0$, which we denote by $v_1$,  that is globally small via a suitable cut-off procedure.
		To do so,  let $\eta(x)=\eta(|x|)$ be a smooth non-negative cut-off function
		satisfying 
		\begin{equation*}
			0\leq \eta (|x|) \leq 1, ~	\eta (|x|) =1 \textrm{ for } |x|\geq 2 \textrm{ and } \eta (|x|) =0  \textrm{ for }  |x|\leq 1.
		\end{equation*}
		We also denote for $R_0>1$ that
		\begin{equation}\label{eq:resacledversion}
			\eta_{R_0}(x)=\eta\left(\frac{x}{{R_0}}\right),
		\end{equation}
		then there exists some universal constant $C$ that
		\begin{equation*}
			|\nabla^k \eta_{R_0}|\leq \frac{C}{R_0^k}, ~ \forall k\in \mathbb{N}.
		\end{equation*}
		We define a modified version of $v_0$ by 
		\begin{equation}\label{eq:defoftildev}
			v_1= \eta_{R_0}	 v_0 +w, 
		\end{equation}
		where $w$ is given by
        \begin{equation}\label{eq:corr}
			\begin{aligned}
				w(y) &=  - \frac{1}{2\pi}\nabla_y \int_{\mathbb{R}^2} (\ln |y-z|)v_0\cdot \nabla  \eta_{R_0}(z)dz
				= -\frac{1}{2\pi} \int_{\mathbb{R}^2} \frac{y-z}{|y-z|^2}v_0 \cdot \nabla  \eta_{R_0}(z)dz.
			\end{aligned}
		\end{equation}
        It follows that $w$ solves 
		\begin{equation*}
			\div w = -v_0\cdot \nabla \eta_{R_0}, 
		\end{equation*}
and hence  $\div v_1=0$.
		
		\begin{remark}
			Such a correction term $w$ is obviously not unique. Typically, it can be constructed with compact support in the annulus $B_{2R_0} \setminus B_{R_0}$ using the Bogovskii formula, given that $v_0 \cdot \nabla \eta_{R_0}$ is smooth and supported in that region (cf. \cite[Theorem III.3.3]{Galdi11}). The representation in \eqref{eq:corr}, however, may not have compact support.
		\end{remark}

		We now give the estimates for the modified   profile 	$v_1= \eta_{R_0}	 v_0 +w$.
		\begin{lemma}[Estimates for $  v_1$]\label{lem:Estimates for modified linear part}
			Let $v_0$ satisfy the estimates in Lemma \ref{lem:lineardecay}.
			The vector field $v_1$ defined in \eqref{eq:defoftildev} is smooth, divergence-free, and
			we have the following decay estimates for $v_1$ and $w$:
			\begin{equation}\label{eq:estv1}
			|   v_1(y)|\leq \frac{C(A,R_0)}{1+|y|} ~\textrm{ and }~	|  \nabla^k v_1(y)|\leq \frac{C(A,\beta,k, R_0)}{1+|y|^{1+\beta}}, ~k\in\mathbb{N}^+, 
			\end{equation}
			and 
			\begin{equation}\label{eq:estw}
				|  \nabla^k w(y)|\leq \frac{C(A,\beta, k, R_0)}{1+|y|^{k+2}}, ~ k=0,1,2\cdots,
			\end{equation}
			where the constant $C(A,\beta, R_0)$ depends only on $A$, $\beta$, $k$ and $R_0$.
			Moreover, for any $\epsilon>0$, we can choose $R_0=R_0(A,\beta, \epsilon)$ in \eqref{eq:resacledversion} large enough so that 
			\begin{equation}\label{eq:smallness}
				\|  v_1\|_{L^{\infty}(\mathbb{R}^2)}\leq \epsilon, ~ \|\nabla   v_1\|_{L^{\infty}(\mathbb{R}^2)} \leq \epsilon.
			\end{equation}
		\end{lemma}
		\begin{proof}
        The smoothness of $v_1$ and $\div v_1=0$ is clear from the construction. We  focus on proving \eqref{eq:estv1}-\eqref{eq:smallness}.
			First of all, it follows Lemma \ref{lem:lineardecay} that
			\begin{equation*}
				| \eta_{R_0}	 v_0|\leq \frac{C(A)}{1+|y|} \quad \text{and} \quad 	| \nabla (\eta_{R_0}	 v_0)|\leq \frac{C(A,\beta)}{1+|y|^{1+\beta}}.
			\end{equation*} 		
			Hence, to prove \eqref{eq:estv1} and \eqref{eq:estw}, it suffices to prove \eqref{eq:estw}. We only  consider the case for $|y|\geq 3R_0$, as the boundedness and smoothness of $w$ inside $|y|\leq 3R_0$ is easy to see using the boundedness and smoothness of $v_0$.
			Since 
			$v_0 \cdot \nabla  \eta_{R_0}$ is smooth and has compact support, we have
			\begin{equation*}
				\int_{\mathbb{R}^2} v_0 \cdot \nabla  \eta_{R_0} = \int_{B_{2R_0}\setminus B_{R_0}} \textrm{ div } (\eta_{R_0} v_0) =  \int_{\partial B_{2R_0}}  v_0\cdot n d\sigma=0.
			\end{equation*}
		Therefore, it holds that
			\begin{equation*}
				\begin{aligned}
					w(y) &=    -\frac{1}{2\pi} \int_{\mathbb{R}^2} \left(  \frac{y-z}{|y-z|^2} - \frac{y}{|y|^2}  \right) v_0 \cdot \nabla  \eta_{R_0}(z) dz\\
					&=-\frac{1}{2\pi} \int_{B_{2R_0}} \left(  \frac{y-z}{|y-z|^2} - \frac{y}{|y|^2}  \right) v_0 \cdot \nabla  \eta_{R_0}(z) dz.
				\end{aligned}	
			\end{equation*}
			We can rewrite the above as
			\begin{equation*}
				\begin{aligned}
					w(y) & = -\frac{1}{2\pi} \int_{B_{2R_0}} \left(  \frac{y-z}{|y-z|^2} - \frac{y}{|y|^2}  \right) v_0 \cdot \nabla  \eta_{R_0}(z) dz \\
					&= -\frac{1}{2\pi} \int_{B_{2R_0}} \left(\int_{0}^{1} \frac{d}{ds}\frac{y-sz}{|y-sz|^2} ds\right) v_0 \cdot \nabla  \eta_{R_0}(z) dz \\
					&= -\frac{1}{2\pi} \int_{B_{2R_0}}  \left(  \int_{0}^{1}(-z) \cdot \nabla \frac{y-sz}{|y-sz|^2} ds \right) v_0 \cdot \nabla  \eta_{R_0}(z) dz.
				\end{aligned}
			\end{equation*}
			Hence, if $|y|\geq 3R_0$, one has
			\begin{equation*}
				\begin{aligned}
					|w(y)| \leq 
					&\frac{1}{2\pi} \int_{B_{2R_0}}    \int_{0}^{1}|z|  \frac{1}{|y-sz|^2} ds  |v_0 \cdot \nabla  \eta_{R_0}|dz	\\
					&\leq \frac{9 R_0}{\pi} \int_{B_{2R_0}}  \frac{1}{|y|^2}|v_0 \cdot \nabla  \eta_{R_0}|dz \leq \frac{C(A,R_0)}{|y|^2},
				\end{aligned}
			\end{equation*}
			where we have used that $|y-sz|\geq |y|-s|z| \geq |y|-|z|\geq\frac{1}{3}|y|$ for $|y|\geq 3R_0$ and $|z|\leq 2R_0$.
			This shows that 
			\begin{equation*}
				|  w(y)|\leq \frac{C(A,R_0)}{1+ |y|^2}.
			\end{equation*}
			The proof for $|  \nabla^k w(y)|\leq \frac{C(A,\beta, k, R_0)}{1+|y|^{k+2}}$, $k\geq 1$, 
            is similar to the above calculation by taking the derivatives into the kernel and using the property $\int_{\mathbb{R}^2} v_0 \cdot \nabla  \eta_{R_0} = 0$,  we omit details.
			
			To prove \eqref{eq:smallness}, we
			first note  due to Lemma \ref{lem:lineardecay} and the definition of $\eta_{R_0}$ that
			\begin{equation*}
				| \eta_{R_0}	 v_0|\leq \frac{C(A)}{1+|y|} \mathbbm{1}_{B_{R_0}^c}, ~
				|\nabla^k (\eta_{R_0}	 v_0)|\leq \frac{C(A,\beta)}{1+|y|^{1+\beta}} \mathbbm{1}_{B_{R_0}^c}, ~ \textrm{ for } k = 1, 2,
			\end{equation*}
			where $\mathbbm{1}_{\Omega}$ is the characteristic function of $\Omega$ and $B_{R_0}^c$ is the complement of $B_{R_0}$.
			It is then easy to see that  for any $r>2$ we have
			\begin{equation*}
				\| \eta_{R_0}v_0\|_{W^{2,r}(\mathbb{R}^2)} \to 0 \textrm{ as } R_0\to +\infty.
			\end{equation*}
			Using the integral formula for $w$ and the Calderon–Zygmund estimates, we have for any $r>2$ that
			\begin{equation*}
				\|w\|_{W^{2,r}(\mathbb{R}^2)}\leq c(r) \| \eta_{R_0}v_0\|_{W^{2,r}(\mathbb{R}^2)}.
			\end{equation*}
			It also follows from Morrey's inequality that for any $r>2$
			\begin{equation*}
				\|w\|_{C^{1,\alpha}(\mathbb{R}^2)}\leq c(r) \|w\|_{W^{2,r}(\mathbb{R}^2)},
			\end{equation*}
			where $\alpha = 1-\frac{2}{r}$.
			Hence, we have
			\begin{equation*}
				\|w\|_{C^{1,\alpha}(\mathbb{R}^2)}\leq 	c(r)\| \eta_{R_0}v_0\|_{W^{2,r}(\mathbb{R}^2)} \to 0 \textrm{ as } R_0\to +\infty.
			\end{equation*}
			Choose $r=3$ in the above, 
			for any $\epsilon>0$, we can then choose $R_0 = R_0(A, \beta,\epsilon)$ large enough such that
			\begin{equation*}
				\|w\|_{C^{1,\frac{1}{3}}(\mathbb{R}^2)}\leq C \| \eta_{R_0}v_0\|_{W^{2,3}(\mathbb{R}^2)}\leq  \frac{\epsilon}{2}.
			\end{equation*}
			This $R_0 = R_0(A, \beta,\epsilon)$ can already be large enough to make 
			\begin{equation*}
				\|\eta_{R_0}	 v_0\|_{C^{1}(\mathbb{R}^2)} \leq\frac{\epsilon}{2}.
			\end{equation*}
			Combining the above two inequalities, we obtain  \eqref{eq:smallness}. 
			This finishes the  proof of the lemma.
		\end{proof}
		
		With this new profile $v_1= \eta_{R_0}	 v_0 +w$, we define the modified version of $v_{\textrm{re}}$ as
		\begin{equation}\label{eq:modre}
			v_{2}=v- v_1,
		\end{equation}
		where $v$ is the solution to \eqref{eq:NSinself-similar}.
		It is easy to find that $v_2$ satisfies
		\begin{equation}\label{eq:eqfordiff}
			\left\{
			\begin{aligned}
				&-\Delta v_2 -\frac{1}{2}v_2 -\frac{1}{2}y\cdot \nabla v_2+  v_1\cdot \nabla v_2 + v_2\cdot \nabla  v_1+v_2\cdot \nabla v_2 +\nabla q_2=- v_1\cdot\nabla  v_1 - F(v_0,w),\\
				&  \div v_2=0,		
			\end{aligned}		
			\right.
		\end{equation}
        in $\mathbb{R}^2$, 
		where using \eqref{eq:equforv1}, we have
		\begin{equation}\label{eq:eqformodifv}
			\begin{aligned}
					-\Delta  v_1 -\frac{1}{2} v_1 -\frac{1}{2}y\cdot \nabla  v_1 & = F(v_0,w) \\
				&	=  	-\Delta (\eta_{R_0}	v_0 +w )- \frac{1}{2}(\eta_{R_0}	v_0 +w ) -\frac{1}{2}y\cdot \nabla (\eta_{R_0}	v_0 +w )\\
				&= - 2 \nabla \eta_{R_0}\nabla 	v_0- \Delta \eta_{R_0}	v_0   -\frac{1}{2}y\cdot \nabla \eta_{R_0}	 v_0 -\Delta w  -\frac{1}{2} w -\frac{1}{2}y\cdot \nabla  w.
			\end{aligned}
		\end{equation}

		To show Theorem \ref{thm:energyestkey}, we first prove a slightly weaker version for $v_2$. It  says that we have $H^1$-estimate and $L^{p}$-estimate for  $v_2$  when $1<p<2 $.  Hence, we also have $H^1$-estimate and  $L^{p}$-estimate for  $v_{\text{re}}$ when $1<p<2 $.
		\begin{theorem}\label{thm:energyest}
        Let the divergence-free, self-similar initial data $u_0(x)\in C^{0,\beta}_{loc}$ be given, where $0<\beta\leq 1$, and let
			$A=\|u_0\|_{C^{0,\beta}(S^1)}<+\infty$.		
			For any $1<p<2 $, we can  choose
			$R_0=R_0(A,\beta, p)$ in the definition of $v_1$  \eqref{eq:defoftildev} large enough, such that 
     the following a priori estimates  hold,
			\begin{equation*}
				\int_{\mathbb{R}^2} |v_2|^2 + |\nabla v_2|^2 dy \leq C(A,\beta, p),
			\end{equation*}
			and 
			\begin{equation*}
				\|v_2\|_{L^{p}(\mathbb{R}^2)} \leq C(A, \beta, p).
			\end{equation*}
			In particular, choose $p=\frac{3}{2}$ and fix a $R_0=R_0(A,\beta)$ 
            yields that  
			\begin{equation*}
				\int_{\mathbb{R}^2} |v_2|^2 + |\nabla v_2|^2 dy \leq C(A, \beta),
			\end{equation*}
          where the constant $C(A, \beta)$ depends only on $A$ and $\beta$.
		\end{theorem}
		
		\begin{proof}
			\textbf{Step 1:  Estimate of $\|\nabla v_2\|_{L^2}$.}
			Multiplying \eqref{eq:eqfordiff} by $v_2$, then integrating both sides, it leads to
			\begin{equation}\label{eq:sum1}
				\begin{aligned}
					\int_{\mathbb{R}^2} |\nabla v_2|^2 dy   + \int_{\mathbb{R}^2}  (v_2\cdot \nabla  v_1)\cdot v_2 dy +\int_{\mathbb{R}^2}  ( v_1\cdot\nabla  v_1)\cdot v_2 d y=-\int_{\mathbb{R}^2} F(v_0,w)v_2.
				\end{aligned}
			\end{equation}
			Multiplying \eqref{eq:eqfordiff} by $ v_1$, then integrating both sides, it leads to
			\begin{equation}\label{eq:sum2}
				\begin{aligned}
					&\int_{\mathbb{R}^2} \partial_i v_{2,j} \partial_i v_{1,j} dy   + \int_{\mathbb{R}^2}- \frac{1}{2} v_1\cdot v_2-\frac{1}{2}(y\cdot \nabla v_2)\cdot  v_1 dy \\
					&~ +
					\int_{\mathbb{R}^2}  ( v_1\cdot \nabla v_2)\cdot  v_1dy +\int_{\mathbb{R}^2}  (v_2\cdot\nabla v_2)\cdot  v_1 d y = -\int_{\mathbb{R}^2} F(v_0,w) v_1.
				\end{aligned}
			\end{equation}	
			Here and in the following $v_{1,j}$ means the $j$-th component of $v_1$,
			$v_{2,j}$ means the $j$-th component of $v_2$ for $j=1,2$.
			Multiplying  Equation \eqref{eq:eqformodifv}  of $ v_1$ by $v_2$, then integrating both sides, it leads to
			\begin{equation}\label{eq:sum3}
				\begin{aligned}
					&	 \int_{\mathbb{R}^2} \partial_i v_{2,j} \partial_i v_{1,j} dy -  \int_{\mathbb{R}^2} \frac{1}{2} v_1\cdot v_2+\frac{1}{2}y_i \partial_i v_{1,j} v_{2,j} dy \\
					=&\int_{\mathbb{R}^2} \partial_i v_{2,j} \partial_i v_{1,j} dy + \int_{\mathbb{R}^2} -\frac{1}{2} v_1\cdot v_2+  v_1\cdot v_2 +
					\frac{1}{2}y_i \partial_i v_{2,j} v_{1,j} dy \\
					=&\int_{\mathbb{R}^2} \partial_i v_{2,j} \partial_i v_{1,j} dy + \int_{\mathbb{R}^2} \frac{1}{2} v_1\cdot v_2 +
					\frac{1}{2}(y\cdot \nabla v_2)\cdot  v_1 dy =\int_{\mathbb{R}^2} F(v_0,w)\cdot v_2,\\
				\end{aligned}
			\end{equation}
            where integration by parts has been used.
			We also note that
			\begin{equation*}
				\int_{\mathbb{R}^2}(v_2 \cdot \nabla v_2)\cdot  v_1 = 	\int_{\mathbb{R}^2}v_{2,i}\partial _i v_{2,j} v_{1,j}dy =- \int_{\mathbb{R}^2}v_{2,i} v_{2,j} \partial _i v_{1,j}dy = - \int_{\mathbb{R}^2 } (v_2\cdot \nabla v_1) \cdot v_2,
			\end{equation*}
			and 
			\begin{equation*}
				\int_{\mathbb{R}^2}( v_1  \cdot \nabla v_2)\cdot  v_1 = 
				\int_{\mathbb{R}^2}{v}_{1,i}\partial _i v_{2,j} v_{1,j}dy =- \int_{\mathbb{R}^2}{v}_{1,i} v_{2,j} \partial _i v_{1,j}dy =
				-\int_{\mathbb{R}^2 } ( v_1\cdot \nabla v_1) \cdot v_2.
			\end{equation*}
			Hence, adding \eqref{eq:sum1}-\eqref{eq:sum3} together, we have
			\begin{equation}\label{eq:enerid}
				\begin{aligned}
					&\int_{\mathbb{R}^2} |\nabla v_2|^2 dy + 2\int_{\mathbb{R}^2} \partial_i v_{2,j} \partial_i v_{1,j} dy=-\int_{\mathbb{R}^2} F(v_0,w) \cdot v_1\\
					= & \int_{\mathbb{R}^2} \left(2 \nabla \eta_{R_0}\nabla 	v_0+\Delta \eta_{R_0}	v_0   +\frac{1}{2}y\cdot \nabla \eta_{R_0}	 v_0 + \Delta w  + \frac{1}{2} w + \frac{1}{2}y\cdot \nabla  w\right)\cdot v_1.
				\end{aligned}
			\end{equation}
			For the right-hand side of above, 
with the help of Lemma \ref{lem:Estimates for modified linear part}, one has
			\begin{equation*}
				\begin{aligned}
					&\int_{\mathbb{R}^2} \left(2 \nabla \eta_{R_0}\nabla 	v_0+\Delta \eta_{R_0}	v_0   +\frac{1}{2}y\cdot \nabla \eta_{R_0}	 v_0 \right)\cdot v_1\\
					\leq 
					& \int_{\mathbb{R}^2} \frac{C(A,R_0)}{1+|y|}\left|2 \nabla \eta_{R_0}\nabla 	v_0 + \Delta \eta_{R_0}	v_0  + \frac{1}{2}y\cdot \nabla \eta_{R_0}	 v_0 \right|\\
					\leq & \int_{B_{2R_0}} \frac{C(A,R_0)}{1+|y|^2} = C(A,R_0),
				\end{aligned}
			\end{equation*}
			and
			\begin{equation*}
				\begin{aligned}
					\int_{\mathbb{R}^2} \left(\frac{1}{2} w +\frac{1}{2}y\cdot \nabla  w\right)\cdot v_1
					\leq	& C(A,R_0)\int_{\mathbb{R}^2} \frac{1}{1+|y|}\left|\frac{1}{2} w +\frac{1}{2}y\cdot \nabla  w\right|\\
					\leq & \int_{\mathbb{R}^2} \frac{C(A,\beta, R_0)}{1+|y|^3} = C(A,\beta,R_0).
				\end{aligned}
			\end{equation*}
			Finally, using Lemma \ref{lem:Estimates for modified linear part} again yields
			\begin{equation*}
				\begin{aligned}
					&\int_{\mathbb{R}^2} \Delta w \cdot v_1 \leq  \int_{\mathbb{R}^2} \frac{C(A,\beta,R_0)}{1+|y|^5} = C(A,\beta, R_0).
				\end{aligned}
			\end{equation*}
			Hence, go back to \eqref{eq:enerid}, we have
			\begin{equation*}
				\int_{\mathbb{R}^2} |\nabla v_2|^2 dy + 2\int_{\mathbb{R}^2} \partial_i v_{2,j} \partial_i v_{1,j} dy=-\int_{\mathbb{R}^2} F(v_0,w) v_1\leq C(A,\beta, R_0).
			\end{equation*}
			This together with  Young's inequality and \eqref{eq:estv1} yields that
			\begin{equation}\label{eq:gradientest}
				\int_{\mathbb{R}^2} |\nabla v_2|^2 dy \leq 4 \int_{\mathbb{R}^2} |\nabla  v_1|^2dy-2\int_{\mathbb{R}^2} F(v_0,w) v_1\leq C(A,\beta, R_0),
			\end{equation}
			where we have used
			\begin{equation}\label{eq:Ditrichelten}
			\int_{\mathbb{R}^2} |\nabla  v_1|^2dy \leq \int_{\mathbb{R}^2}  \frac{C(A,\beta, R_0)}{1+|y|^{2+2\beta}} dy \leq C(A, \beta, R_0).
			\end{equation}
			This is our initial Dirichlet energy estimate. 
			
			\textbf{Step 2: Estimate of 	$\|v_2\|_{L^{p}}$.}
			We multiply the equation of $v_2$ \eqref{eq:eqfordiff} by $|v_2|^{p-2}v_2$, where $1<p<2$, and integrate on both sides, we have
			\begin{equation}\label{eq:muptiplier1}
				\begin{aligned}
					& \int_{\mathbb{R}^2} \left(	-\Delta v_2 -\frac{1}{2}v_2 -\frac{1}{2}y\cdot \nabla v_2+  v_1\cdot \nabla v_2 +v_2\cdot \nabla v_2  + v_2\cdot \nabla  v_1 +\nabla q_2 \right) \cdot |v_2|^{p-2}v_2\\
					= &-\int_{\mathbb{R}^2}  \left( v_1\cdot\nabla  v_1 - F(v_0,w) \right)\cdot  |v_2|^{p-2}v_2.	
				\end{aligned}
			\end{equation}
			We consider the left-hand side of \eqref{eq:muptiplier1} first.
            
            \textbf{Estimates for the first five terms on the left-hand side of \eqref{eq:muptiplier1}.}
            The estimates of the first five terms on the left-hand side of \eqref{eq:muptiplier1} are quite straightforward using integration by parts.
			For the first term, we have 
			\begin{equation*}
				\begin{aligned}
					&\int_{\mathbb{R}^2}	-\Delta v_2  \cdot |v_2|^{p-2}v_2 = \int_{\mathbb{R}^2}	-\Delta v_{2,i}   |v_2|^{p-2}v_{2,i}\\
					= &\int_{\mathbb{R}^2}	\nabla v_{2,i}  \cdot |v_2|^{p-2}\nabla v_{2,i} + \partial_j v_{2,i} (p-2) |v_2|^{p-3} \partial_j \sqrt{v_{2,1}^2+v_{2,2}^2}v_{2,i}\\
					=&\int_{\mathbb{R}^2} |\nabla v_2|^2 |v_2|^{p-2} + \partial_j v_{2,i} (p-2) |v_2|^{p-3} \frac{v_{2,1}\partial_j  v_{2,1} + v_{2,2}\partial_j  v_{2,2}  }{\sqrt{v_{2,1}^2+v_{2,2}^2}}v_{2,i}\\
					=&\int_{\mathbb{R}^2} |\nabla v_2|^2 |v_2|^{p-2} +   (p-2) |v_2|^{p-4} (v_{2,1}\partial_j  v_{2,1} + v_{2,2}\partial_j  v_{2,2}  ) v_{2,i}  \partial_j v_{2,i}\\
					=&\int_{\mathbb{R}^2} |\nabla v_2|^2 |v_2|^{p-2} +  \sum_{j=1}^{2}(p-2)|v_2|^{p-4}  (v_{2,i}  \partial_j v_{2,i})^2\\
					\geq & \int_{\mathbb{R}^2} |\nabla v_2|^2 |v_2|^{p-2} + (p-2) |v_2|^{p-2} |\nabla v_2|^2= \int_{\mathbb{R}^2} (p-1) |\nabla v_2|^2 |v_2|^{p-2},
				\end{aligned}
			\end{equation*}
			where  the condition $p<2$ guarantees  the second-to-last inequality holds.
			In above, $v_{2,i}$ is the $i$-th component of $v_2$ and $ \partial_j v_{2,i} = \frac{\partial  v_{2,i}}{\partial x_j}$ for $i,j=1,2$. Next, we have 
			\begin{equation*}\label{eq:est11}
				\begin{aligned}
					\int_{\mathbb{R}^2}	 -\frac{1}{2}y\cdot \nabla v_2\cdot |v_2|^{p-2}v_2
					&=\int_{\mathbb{R}^2}-\frac{1}{4}y\cdot \nabla |v_2|^2 (|v_2|^{2})^{\frac{p}{2}-1}\\
					&=\int_{\mathbb{R}^2}-\frac{1}{4}y\cdot \nabla  (|v_2|^{2})^{\frac{p}{2}}\frac{2}{p}\\
					&=\frac{1}{p} \int_{\mathbb{R}^2}  |v_2|^{p}.\\
				\end{aligned}
			\end{equation*}
		Using integration by parts and $\div v_1=0$ yields that
			\begin{equation*}\
				\begin{aligned}
					\int_{\mathbb{R}^2}	 v_1\cdot \nabla v_2\cdot |v_2|^{p-2}v_2
					&= \frac{1}{2}\int_{\mathbb{R}^2}	 v_1\cdot \nabla(|v_2|^{2})^{\frac{p}{2}}\frac{2}{p}=0.
				\end{aligned}
			\end{equation*}
			Similarly, one has
			\begin{equation*}\
				\begin{aligned}
					\int_{\mathbb{R}^2}	 v_2\cdot \nabla v_2\cdot |v_2|^{p-2}v_2
					=0.
				\end{aligned}
			\end{equation*}
			Combining all the above estimates gives
			\begin{equation}\label{eq:est1}
				\begin{aligned}
					&\int_{\mathbb{R}^2}	(-\Delta v_2 -\frac{1}{2}v_2 -\frac{1}{2}y\cdot \nabla v_2+  v_1\cdot \nabla v_2+v_2\cdot \nabla v_2 ) \cdot |v_2|^{p-2}v_2\\
					\geq& \int_{\mathbb{R}^2} (p-1)|\nabla v_2|^2 |v_2|^{p-2} + \left(\frac{1}{p}-\frac{1}{2}\right) |v_2|^{p}.
				\end{aligned}
			\end{equation}
			It is important to note that $p-1>0$ and $\frac{1}{p}-\frac{1}{2}>0$ provided $1<p<2$.

              \textbf{Estimates for the last two terms on left-hand side of \eqref{eq:muptiplier1}.}
			For the second last term on the left-hand side of \eqref{eq:muptiplier1}, we have
			\begin{equation}\label{eq:potentialterm}
				\begin{aligned}
					&\left|\int_{\mathbb{R}^2}	(v_2\cdot\nabla  v_1) \cdot |v_2|^{p-2}v_2\right|\leq \|\nabla  v_1\|_{L^{\infty}} \int_{\mathbb{R}^2} |v_2|^{p}.
				\end{aligned}
			\end{equation}
			To estimate the last term on the left-hand side of \eqref{eq:muptiplier1},	we first note that
			\begin{equation}
				\begin{aligned}
					-\Delta q_2 &= \textrm{ div } ( v_1\cdot \nabla v_2 + v_2\cdot \nabla  v_1 + v_2\cdot \nabla v_2 +  v_1\cdot\nabla v_1) -\div  F(v_0,w)\\
					&= \textrm{ div } \div (( v_1+v_2) \otimes( v_1+v_2) ),
				\end{aligned}
			\end{equation}
			since we have from  \eqref{eq:eqformodifv} that
			\begin{equation*}
				\div  F(v_0,w) = \div\left(-\Delta  v_1 -\frac{1}{2} v_1 -\frac{1}{2}y\cdot \nabla  v_1\right)=0,
			\end{equation*}
            where  $\div v_1=0$ has been used.
			It follows from the  Calderon–Zygmund estimates (\cite[Chapter V]{Stein93}) that there exists a constant $c(p)$ such that
			\begin{equation}\label{eq:potenest}
				\begin{aligned}
					\|\nabla q_2\|_{L^{p}}&\leq c(p)	  \left(\| v_1\nabla  v_1\|_{L^{p}}+\|v_2\nabla  v_1\|_{L^{p}}\right. \\
					&\qquad \qquad + \left. \| v_1 \nabla v_2\|_{L^{p}}+\|v_2\nabla v_2\|_{L^{p}} \right).
				\end{aligned}
			\end{equation}
			We now fix some $p\in(1,2)$,  let $\epsilon$ in Lemma \ref{lem:Estimates for modified linear part} be small enough, say
			\begin{equation*}
				\epsilon \leq \min\left\{\frac{1}{10c(p)},\frac{1}{10}\right\}
				\cdot\left(\frac{1}{p}-\frac{1}{2}\right),
			\end{equation*}
			by choosing 
			\begin{equation*}
				R_0=R_0(A,\beta, \epsilon)  =R_0(A, \beta, p) 
			\end{equation*}
			in \eqref{eq:defoftildev} large enough.
			It follows that 
			\begin{equation*}
				\| v_1\|_{L^\infty}, ~	\|\nabla v_1\|_{L^\infty}\leq \epsilon \leq \min\left\{\frac{1}{10c(p)},\frac{1}{10}\right\}
				\cdot\left(\frac{1}{p}-\frac{1}{2}\right).
			\end{equation*}
			 We also have from \eqref{eq:gradientest} that
			\begin{equation*}
				\int_{\mathbb{R}^2} |\nabla v_2|^2 dy \leq 4 \int_{\mathbb{R}^2} |\nabla  v_1|^2dy-2\int_{\mathbb{R}^2} F(v_0,w) v_1\leq C(A,\beta, p).
			\end{equation*}
			We  can now estimate the last term on the left-hand side of \eqref{eq:muptiplier1} using \eqref{eq:potenest} as follows:
			\begin{equation}\label{eq:pressterm}
				\begin{aligned}
					& \quad \int_{\mathbb{R}^2} (\nabla q_2) \cdot |v_2|^{p-2} v_2 \le \|v_2\|_{L^{p}}^{p-1} \|\nabla q_2\|_{L^{p}} \\
					& \le c(p)\|v_2\|_{L^{p}}^{p-1} \left( \| v_1 \nabla  v_1\|_{L^{p}} + \|v_2 \nabla  v_1\|_{L^{p}} + \| v_1 \nabla v_2\|_{L^{p}} + \|v_2 \nabla v_2\|_{L^{p}}  \right) \\
					& \le c(p)\|v_2\|_{L^{p}}^{p-1} \left( C(A,\beta, p) + \frac{1}{10c(p)} \left(\frac{1}{p}-\frac{1}{2}\right) \|v_2\|_{L^{p}} + \| v_1\|_{L^{\frac{2p}{2-p}}} \|\nabla v_2\|_{L^2} + \|\nabla v_2\|_{L^2} \|v_2\|_{L^{\frac{2p}{2-p}}} \right) \\
					& \le c(p)\|v_2\|_{L^{p}}^{p-1} \left( C(A,\beta, p) + \frac{1}{10c(p)} \left(\frac{1}{p}-\frac{1}{2}\right)\|v_2\|_{L^{p}} + \| v_1\|_{L^{\frac{2p}{2-p}}} \|\nabla v_2\|_{L^2} +c_1(p) \|\nabla v_2\|_{L^2}^{1+\frac{p}{2}} \|v_2\|_{L^{p}}^{1-\frac{p}{2}} \right) \\
					& \leq c(p) \|v_2\|_{L^{p}}^{p-1} \left( C(A,\beta, p) + \frac{1}{10c(p)}\left(\frac{1}{p}-\frac{1}{2}\right) \|v_2\|_{L^{p}} + C(A,\beta, p) + C(A,\beta, p) \|v_2\|_{L^{p}}^{1-\frac{p}{2}} \right) \\
					& \le C(A,\beta, p) + \frac{1}{5} \left(\frac{1}{p}-\frac{1}{2}\right) \|v_2\|_{L^{p}}^{p},
				\end{aligned}
			\end{equation}
			where we have used Gagliardo–Nirenberg interpolation inequality 
			\begin{equation*}
				\|v_2\|_{L^{\frac{2p}{2-p}}(\mathbb{R}^2)} \leq c_1(p)	\|\nabla v_2\|_{L^{2}(\mathbb{R}^2)}^{\frac{p}{2}} 	\|v_2\|_{L^{p}(\mathbb{R}^2)}^{1-\frac{p}{2}}, ~ 1<p<2,
			\end{equation*}
			and also the fact that for any $1<p<2$,
			\begin{equation*}
				\| v_1\|_{L^{\frac{2p}{2-p}}} \leq C(A,\beta, p), ~	| v_1\nabla  v_1|\leq \frac{C(A,\beta, R_0)}{1+|y|^2}=\frac{C(A,\beta, p)}{1+|y|^2}, ~\| v_1 \nabla  v_1\|_{L^{p}}\leq C(A,\beta, p),
			\end{equation*}
			which follows from Lemma \ref{lem:Estimates for modified linear part}.
			Equation \eqref{eq:potentialterm} becomes 
			\begin{equation}\label{eq:potentialterm1}
				\begin{aligned}
					&\left|\int_{\mathbb{R}^2}	(v_2\cdot\nabla  v_1) \cdot |v_2|^{p-2}v_2\right|\leq \frac{1}{10} \left(\frac{1}{p}-\frac{1}{2}\right) \int_{\mathbb{R}^2} |v_2|^{p}.
				\end{aligned}
			\end{equation}
            
              \textbf{Estimates for the right-hand side of \eqref{eq:muptiplier1}.}
			We now turn to estimate  the right-hand side of \eqref{eq:muptiplier1}.
			We  claim that for any $1<p<2$,
			\begin{equation}\label{eq:forcetermest}
				\begin{aligned}
					\|F(v_0,w)	\|_{L^{p}}\leq C(A,\beta, p).
				\end{aligned}
			\end{equation}
			Indeed,  it follows from  Lemma \ref{lem:Estimates for modified linear part}  that for any $1<p<2$, one has
			\begin{equation*}
				\int_{\mathbb{R}^2} |\Delta w|^{p} \\
				\leq C(A,\beta, p),
			\end{equation*}
			and 	that
			\begin{equation*}
				\left|2 \nabla \eta_{R_0}\nabla 	v_0+\Delta \eta_{R_0}	v_0   +\frac{1}{2}y\cdot \nabla \eta_{R_0}	 v_0  + \frac{1}{2} w + \frac{1}{2}y\cdot \nabla  w\right|\leq \frac{C(A,\beta, p)}{1+|y|^2}.
			\end{equation*}
           Hence,  using \eqref{eq:eqformodifv}, we have for any $1<p<2$ that
			\begin{equation*}\label{eq:forceterms}
				\begin{aligned}
	\int_{\mathbb{R}^2}	|F(v_0,w)|^{p}
					\leq & C\int_{\mathbb{R}^2} \left|2 \nabla \eta_{R_0}\nabla 	v_0+\Delta \eta_{R_0}	v_0   +\frac{1}{2}y\cdot \nabla \eta_{R_0}	 v_0  + \frac{1}{2} w + \frac{1}{2}y\cdot \nabla  w\right|^{p} 
					 	+
					\int_{\mathbb{R}^2} |\Delta w|^{p}\\
                    \leq &
					C(A,\beta, p),
				\end{aligned}
			\end{equation*}	
		Now we can conclude by using \eqref{eq:forcetermest} that
			\begin{equation}\label{eq:force}
				\begin{aligned}
					&\quad	\int_{\mathbb{R}^2}	( v_1\cdot\nabla  v_1 - F(v_0,w))\cdot |v_2|^{p-2}v_2 \\
					&\leq \int_{\mathbb{R}^2}	 C| v_1\cdot\nabla  v_1|^{p} + C | F(v_0,w)|^{p} + \frac{1}{10}\left(\frac{1}{p}-\frac{1}{2}\right)|v_2|^{p}  \\
					& \leq C(A,\beta, p)+  \frac{1}{10}\left(\frac{1}{p}-\frac{1}{2}\right)\int_{\mathbb{R}^2} |v_2|^{p}.
				\end{aligned}
			\end{equation}
			Combining  \eqref{eq:est1}, \eqref{eq:pressterm}, \eqref{eq:potentialterm1} and \eqref{eq:force}, we have 
			\begin{equation*}
				\int_{\mathbb{R}^2} |\nabla v_2|^2 |v_2|^{p-2} +  |v_2|^{p} \leq C(A,\beta, p).
			\end{equation*}
			The application of  the Gagliardo–Nirenberg interpolation inequality  for any fixed $1<p<2$ also gives that
			\begin{equation*}
				\|v_2\|_{L^{2}(\mathbb{R}^2)} \leq 	\|v_2\|_{L^{p}(\mathbb{R}^2)}^{\frac{p}{2}} 	\|\nabla v_2\|_{L^{2}(\mathbb{R}^2)}^{1-\frac{p}{2}} \leq C(A,\beta, p).
			\end{equation*}					
			This finishes the proof.
		\end{proof}
		
		\begin{remark}
			Throughout the proof, the finiteness of the  Dirichlet energy  of $v_1$ (or equivalently Dirichlet energy  of $v_0$) \eqref{eq:Ditrichelten} is the only place where we require the H\"older continuity exponent $\beta>0$. 
		\end{remark} 
        We indeed have proved $H^1$-estimate \eqref{eq:energyest} and $L^p$-estimate of $v_{\textrm{re}}$ when $1<p\leq 2$. To prove the rest of 
Theorem \ref{thm:energyestkey}, we use the $H^2$-estimate of $v_{\text{re}}$ to be proved in Proposition \ref{pro:H2est}. The proof of 
Proposition \ref{pro:H2est} only needs $H^1$-estimate \eqref{eq:energyest}  of $v_{\textrm{re}}$.
		\begin{proof}[Proof of Theorem \ref{thm:energyestkey} ]
			Using the standard bootstrapping argument, along
			with the regularity estimates of the Stokes equations, one can show that $v_{\text{re}}\in H^1(\mathbb{R}^2)$ is indeed smooth. For example, one can refer to \cite[Theorem 3.8]{Yang26} for a detailed presentation of the bootstrapping argument.
			This also implies the smoothness of $u(x,1)$ as $e^{\Delta} u_0$ is already smooth.	We next prove the corresponding estimates for $v_{\text{re}}$.	Based on the definitions of $v_{\textrm{re}}$ and $v_2$ in \eqref{eq:deforre} and \eqref{eq:modre} respectively, we have
			\begin{equation*}
				v_{\textrm{re}} - v_2 =v_1 -v_0 = (\eta_{R_0}-1)v_0 +w.
			\end{equation*}
			It then follows from Lemmas \ref{lem:lineardecay} and  \ref{lem:Estimates for modified linear part} and 
			Theorem \ref{thm:energyest} that for $1<p<2$, 
			\begin{equation}\label{eq:remainest}
				\|v_{\text{re}} \|_{L^{p}(\mathbb{R}^2)}\leq C(A,\beta, p).
			\end{equation}
			and  
			\begin{equation*}
				\|v_{\text{re}} \|_{H^1(\mathbb{R}^2)}\le C(A,\beta).
			\end{equation*}
            
            To show higher integrability, we show that the $H^1$-estimates for $v_{\textrm{re}}$ can be improved to $H^2(\mathbb{R}^2)$, which will be done in Proposition \ref{pro:H2est} below.
			It follows that
			\begin{equation}\label{eq:Linf}
				\|v_{\text{re}} \|_{L^{\infty}(\mathbb{R}^2)}\le C(A,\beta).
			\end{equation}
			Now, \eqref{eq:Linf} combines with  \eqref{eq:remainest} and the interpolation inequality implies that for ant $p>1$
			\begin{equation*}
				\|v_{\text{re}} \|_{L^{p}(\mathbb{R}^2)}\leq C(A,\beta, p).
			\end{equation*}
			Using the above and noting that
			\begin{equation*}
				\Delta 	q = \operatorname{div} (v_{\text{re}} \cdot \nabla v_{\text{re}} + v_0 \cdot \nabla v_{\text{re}} + v_{\text{re}} \cdot \nabla v_0 + v_0 \cdot \nabla v_0)= \textrm{ div } \div (( v_0 +v_{\text{re}} ) \otimes( v_0 +v_{\text{re}} ) ),
			\end{equation*}
			the  Calderon–Zygmund  estimates and the estimates for $v_0$ and $v_{\text{re}}$ give that for any $p>1$ 
			\begin{equation*}
				\|q \|_{L^{p}(\mathbb{R}^2)}\leq C(A,\beta, p).
			\end{equation*}
			Similarly, the Calderon–Zygmund  estimates \eqref{eq:potenest}, together with 
            the above estimates of $v_0$ and $v_{\text{re}}$,
            yield that for any $p>1$ 
			\begin{equation*}
				\|\nabla q \|_{L^{p}(\mathbb{R}^2)}\le C(A,\beta,p).
			\end{equation*}
            This finishes the proof of Theorem \ref{thm:energyestkey}.
		\end{proof}


		\section{Weighted Energy Estimates and Pointwise Decay}\label{sec:pointwise}
		In this section, we establish the weighted energy estimate and the pointwise decay estimates for the remainder part $v_{\text{re}}$. For the three-dimensional case, Jia and \v{S}ver\'{a}k \cite{Jia14} first utilized localized smoothing estimates for local Leray solutions to derive pointwise decay. Here, we adopt the 
		methods developed in  \cite{Lai19}. We first give the weighted energy estimates that $\sqrt{1+|y|^2}v_{\text{re}}$ in $H^2(\mathbb{R}^2)$. This gives a first pointwise decay via Sobolev embeddings, i.e.,
		\begin{equation*}
			|v_{\text{re}} (y)| \leq \frac{C}{	1 + |y| } \textrm{ for all } y \in \mathbb{R}^2.
		\end{equation*}
		With this decay estimate in hand, we can improve the estimates  by  using the linear theory of Stokes and the heat equation to have that
		\begin{equation*}
			|v_{\text{re}} (y)| \leq \frac{C}{	1 + |y|^2}\textrm{ for all } y \in \mathbb{R}^2.
		\end{equation*}
		Furthermore, we  prove the optimal decay for $v_{\text{re}}$ if the initial data $u_0\in C^2_{loc}$.

		In light of the global energy estimates established in Theorem \ref{thm:energyestkey}, we do not need to work for $v_{2}$ and $v_1$ anymore,
		which are modified versions of $v_{\text{re}}$ and $v_0$. Hence, in this section, we directly work on $v_{\text{re}}$ and $v_0$. 
		For convenience, we recall that the governing equation for the remainder $v_{\text{re}}$ is given by:
		\begin{equation}\label{eq:differencereca}
			\left\{
			\begin{aligned}
				&-\Delta v_{\text{re}} -\frac{1}{2}v_{\text{re}} -\frac{1}{2}y\cdot \nabla v_{\text{re}}+  v_0\cdot \nabla v_{\text{re}} + v_{\text{re}}\cdot \nabla  v_0 + v_{\text{re}}\cdot \nabla v_{\text{re}} +\nabla q=- v_0\cdot\nabla  v_0, \\
				&  \div v_{\text{re}}=0,
			\end{aligned}
			\right. \textrm{ on } \mathbb{R}^2.
		\end{equation}
		And we also recall that $v_0$ is defined in \eqref{eq:defofv1} and solves \eqref{eq:equforv1}.		
		\subsection{Weighted $H^1$-Estimates of $v_{\text{re}}$}
		As a first step towards the higher-order estimates, we establish a weighted $H^1$ bound for the remainder term $v_{\text{re}}$. It can be seen in the proof that the convection term $-\frac{1}{2}y\cdot \nabla v_{\text{re}}$ plays an important role.
		\begin{proposition}[Weighted $H^1$ Estimate] \label{prop:weighted_H1est}
            Under the same assumption of Theorem \ref{thm:energyestkey},
			consider the corresponding solution $(v_{\text{re}}, q) \in X \times L^2(\mathbb{R}^2)$  to
			\eqref{eq:differencereca}, then 
			we have
			\begin{equation} \label{eq:weighted_H1_bound}
				\| \sqrt{1+|y|^2} v_{\text{re}} \|_{H^1(\mathbb{R}^2)} \le C(A,\beta).
			\end{equation}
		\end{proposition}
		
		\begin{proof}
        Throughout this proof, we use $C$ to denote a constant depending on $A$ and $\beta$ that may be different from line to line.
			We first note that
			\begin{equation*}
				|\nabla ( \sqrt{1+|y|^2} v_{\text{re}}) |\leq \frac{|y|}{\sqrt{1+|y|^2}} |v_{\text{re}}| + \sqrt{1+|y|^2} |\nabla v_{\text{re}}|.
			\end{equation*}
			Hence it holds that
			\begin{equation*}
				|\nabla ( \sqrt{1+|y|^2} v_{\text{re}}) |^2 \leq 2|v_{\text{re}}|^2 +  2(1+|y|^2)|\nabla v_{\text{re}}|^2.
			\end{equation*}
			As we  already have the $H^1$-estimates of $v_{\text{re}}$ due to Theorem \ref{thm:energyestkey}, it suffices to prove the $H^1$-estimates of $|y|v_{\text{re}}$. The key is to  test \eqref{eq:differencereca} with $\phi = |y|^2 v_{\text{re}}$.
			
			For a more rigorous derivation, one can use a regularized weight $h_\epsilon(y) = \frac{|y|}{(1+\epsilon|y|^2)^{1/2}}$ and 
			test \eqref{eq:differencereca} with $\phi = |h_\epsilon|^2 v_{\text{re}}$, then
			pass to the limit $\epsilon \to 0$. For the sake of clarity, we present the estimate using the weight $h(y) = |y|$ and the proof using the weight $h_\epsilon(y)$ follows analogously with minor modification. 
			
			\textbf{Step 1: The Dissipation Term.}
			We calculate the term $-\int \Delta v_{\text{re}} \cdot (|y|^2 v_{\text{re}})$. We have by direct calculations that
			\begin{align*}
				-\int_{\mathbb{R}^2} \Delta v_{\text{re}} \cdot (|y|^2 v_{\text{re}}) \, dy 
				&= \int_{\mathbb{R}^2} \nabla v_{\text{re}} : \nabla(|y|^2 v_{\text{re}}) \, dy \\
				&= \int_{\mathbb{R}^2} \nabla v_{\text{re}} : (|y|^2 \nabla v_{\text{re}} + 2y \otimes v_{\text{re}}) \, dy \\
				&= \int_{\mathbb{R}^2} |y|^2 |\nabla v_{\text{re}}|^2 \, dy + \int_{\mathbb{R}^2} y \cdot \nabla(|v_{\text{re}}|^2) \, dy \\
				&= \int_{\mathbb{R}^2} |y|^2 |\nabla v_{\text{re}}|^2 \, dy - \int_{\mathbb{R}^2} (\nabla \cdot y) |v_{\text{re}}|^2 \, dy \\
				&= \| |y| \nabla v_{\text{re}} \|_{L^2}^2 - 2 \|v_{\text{re}}\|_{L^2}^2\\
                &\geq \| |y| \nabla v_{\text{re}} \|_{L^2}^2-C,
			\end{align*}
            where we have used \eqref{eq:energyest}.

			\textbf{Step 2: The Drift and Linear Damping Terms.}
			The linear drift and  damping part of the operator is $\mathcal{L}(v_{\text{re}}) = -\frac{1}{2}v_{\text{re}} - \frac{1}{2} y \cdot \nabla v_{\text{re}}$. Testing this with $|y|^2 v_{\text{re}}$, we compute that
			\begin{equation}
				I_{drift} = \int_{\mathbb{R}^2} \left( -\frac{1}{2} v_{\text{re}} - \frac{1}{2} y \cdot \nabla v_{\text{re}} \right) \cdot (|y|^2 v_{\text{re}}) \, dy.
			\end{equation}
			The first part is simply $-\frac{1}{2} \int |y|^2 |v_{\text{re}}|^2 dy$. For the second part, we use integration by parts. Note that $\nabla v_{\text{re}} \cdot v_{\text{re}} = \frac{1}{2} \nabla (|v_{\text{re}}|^2)$. Thus,
			\begin{align*}
				-\frac{1}{2} \int_{\mathbb{R}^2} (y \cdot \nabla v_{\text{re}}) \cdot (|y|^2 v_{\text{re}}) \, dy 
				&= -\frac{1}{4} \int_{\mathbb{R}^2} y \cdot \nabla (|v_{\text{re}}|^2) |y|^2 \, dy \\
				&= \frac{1}{4} \int_{\mathbb{R}^2} |v_{\text{re}}|^2 \nabla \cdot (y |y|^2) \, dy.
			\end{align*}
			In 2D, we calculate the divergence explicitly to get
			\begin{equation*}
			\nabla \cdot (y |y|^2) = (\nabla \cdot y)|y|^2 + y \cdot \nabla(|y|^2) = 2|y|^2 + y \cdot (2y) = 2|y|^2 + 2|y|^2 = 4|y|^2.    
			\end{equation*}
			Substituting this back, we have
			\begin{equation*}
			    -\frac{1}{2} \int_{\mathbb{R}^2} (y \cdot \nabla v_{\text{re}}) \cdot (|y|^2 v_{\text{re}}) \, dy = \frac{1}{4} \int_{\mathbb{R}^2} |v_{\text{re}}|^2 (4|y|^2) \, dy = \int_{\mathbb{R}^2} |y|^2 |v_{\text{re}}|^2 \, dy.
			\end{equation*}
			Combining with the damping term $-\frac{1}{2} v_{\text{re}}$ yields
			\begin{equation} \label{eq:drift_coercivity}
				I_{drift} = -\frac{1}{2} \int_{\mathbb{R}^2} |y|^2 |v_{\text{re}}|^2 + \int_{\mathbb{R}^2} |y|^2 |v_{\text{re}}|^2 = \frac{1}{2} \int_{\mathbb{R}^2} |y|^2 |v_{\text{re}}|^2 \, dy.
			\end{equation}
			This positivity is specific to the special structure  of the Leray equations and ensures the weighted $L^2$ norm is controlled.

			\textbf{Step 3: Interaction Terms involving $v_0$.}
			We estimate the term involving $v_0$.   We have
			\begin{equation*}
				\begin{aligned}
					I_{int} &= \int_{\mathbb{R}^2} (v_0\cdot\nabla v_{\text{re}} + v_{\text{re}} \cdot \nabla v_0) \cdot (|y|^2 v_{\text{re}}) dy\\
					&= -\int_{\mathbb{R}^2} v_0\cdot y|v_{\text{re}}|^2+v_{\text{re}}\cdot\nabla (|y|^2v_{\text{re}})\cdot v_0 dy.
				\end{aligned}
			\end{equation*}
			By Lemma \ref{lem:lineardecay}, $| v_0(y)| \le C(1+|y|)^{-1}$.
			Hence, together with \eqref{eq:energyest}, we have
            \begin{equation*}
				\begin{aligned}
					|I_{int}| &\leq \int_{\mathbb{R}^2} |v_0\cdot y|  \cdot|v_{\text{re}}|^2 + 10|v_{\text{re}}| \cdot|y|\cdot|v_{\text{re}}| \cdot|v_0| + 10|v_{\text{re}}|\cdot|y|^2 \cdot |\nabla v_{\text{re}}|\cdot |v_0|\\
					& \leq C \int_{\mathbb{R}^2}   |v_{\text{re}}|^2 + 10|v_{\text{re}}| \cdot |y| \cdot |\nabla v_{\text{re}}|\\
					&\leq C\int_{\mathbb{R}^2}   (|v_{\text{re}}|^2 +  |\nabla v_{\text{re}}|^2) dy  + \frac{1}{10}\int_{\mathbb{R}^2}  |y|^2| v_{\text{re}}|^2 dy\\
					&\leq C + \frac{1}{10}\int_{\mathbb{R}^2}  |y|^2| v_{\text{re}}|^2 dy.
				\end{aligned}
			\end{equation*}

			\textbf{Step 4: Nonlinear term $v_{\text{re}}\cdot \nabla v_{\text{re}}$ and pressure term.}
			For the nonlinear term $v_{\text{re}}\cdot \nabla v_{\text{re}}$, we have similarly as above that
			\begin{equation*}
				\begin{aligned}
					\int_{\mathbb{R}^2} (v_{\text{re}}\cdot \nabla v_{\text{re}})\cdot |y|^2 v_{\text{re}}& = 	\frac{1}{2}\int_{\mathbb{R}^2}|y|^2   v_{\text{re}}\cdot \nabla |v_{\text{re}}|^2 = -	\int_{\mathbb{R}^2}(y\cdot v_{\text{re}}) |v_{\text{re}}|^2\\
					&\leq  C  \int_{\mathbb{R}^2}  | v_{\text{re}}|^4 dy + \frac{1}{10}\int_{\mathbb{R}^2}  |y|^2| v_{\text{re}}|^2 dy\\
					&\leq C  \|v_{\text{re}}\|_{H^1(\mathbb{R}^2)}^4 + \frac{1}{10}\int_{\mathbb{R}^2}  |y|^2| v_{\text{re}}|^2 dy\\
					&\leq  C + \frac{1}{10}\int_{\mathbb{R}^2}  |y|^2| v_{\text{re}}|^2 dy.
				\end{aligned}
			\end{equation*}
			
			We have
			\begin{equation*}
				\| q\|_{L^{2}}\leq C	  \left(\|v_0 \|_{L^{4}} + \|v_{\text{re}}\|_{L^{4}}
				\right)^2 \leq C	  \left(\|v_0 \|_{L^{4}} + \|v_{\text{re}}\|_{H^{1}}
				\right)^2  \leq C.
			\end{equation*}
			Hence, using H\"older inequality, we have
			\begin{equation*}
				\begin{aligned}
					\int_{\mathbb{R}^2} \nabla q \cdot |y|^2 v_{\text{re}}& = 	-\int_{\mathbb{R}^2} 2q v_{\text{re}}\cdot y\leq \frac{1}{10}\int_{\mathbb{R}^2} |y|^2|v_{\text{re}}|^2 + C.
				\end{aligned}
			\end{equation*}
			
			\textbf{Step 5: The source terms.}
            Using Lemma \ref{lem:lineardecay} and \eqref{eq:energyest},
            we have
			\begin{equation*}
				\begin{aligned}
					\quad \int_{\mathbb{R}^2} (-v_0\cdot \nabla v_0 ) \cdot |y|^2 v_{\text{re}}
					& \leq C \int_{\mathbb{R}^2} |\nabla v_0|^2 
					+ \frac{1}{10} \int_{\mathbb{R}^2}  |y|^2|v_{\text{re}}|^2 \\
					& \leq C+  \frac{1}{10} \int_{\mathbb{R}^2}  |y|^2|v_{\text{re}}|^2,
				\end{aligned}
			\end{equation*}
			where we have used that $\int_{\mathbb{R}^2} |\nabla v_0|^2  \leq C$.

			Combining Steps 1-5 gives
			\begin{equation*}
				\begin{aligned}
				\|\nabla(|y|  v_{\text{re}})\|_{L^2}^2 +  \| |y| v_{\text{re}} \|_{L^2}^2\leq 	\||y| \nabla v_{\text{re}}\|_{L^2}^2 +  \| |y| v_{\text{re}} \|_{L^2}^2 + \| v_{\text{re}} \|_{L^2}^2 \leq C.
				\end{aligned}
			\end{equation*}
			This proves the proposition.
		\end{proof}
		
		\subsection{Weighted $H^2$-Estimates of $v_{\textnormal{re}}$}
		With the weighted $H^1$ bounds established in Proposition \ref{prop:weighted_H1est}, we now seek higher-order regularity for the remainder $v_{\text{re}}$. In the two-dimensional setting, $H^1$ regularity is insufficient to ensure a pointwise estimate. And it is necessary to establish estimates in the weighted Sobolev space $H^2$. The main goal of this section is to prove the $H^2$-estimate for $\sqrt{1+|y|^2}v_{\text{re}}$, which implies that $\sqrt{1+|y|^2}v_{\text{re}}(y)$ is bounded.
		We begin with a lemma on  the improved regularity of a general weak solution in the usual Sobolev space.
		
		\begin{lemma}[$H^2$-estimates]\label{eq:H1est}
			Let $f(y) \in L^{2}(\mathbb{R}^{2})$ be given. Assume that $(V, P)$ is a weak solution to the following system:
			\begin{equation*}
            \left\{
				\begin{aligned}
				    &-\Delta V + V\cdot \nabla V - \dfrac{1}{2}(y \cdot \nabla V + V) + \nabla P = f  , \\
					&\operatorname{div} V = 0,           
				\end{aligned}
                  \right. \textrm{ in } \mathbb{R}^{2}.
			\end{equation*}
			Specifically, $V \in H^1(\mathbb{R}^{2})$, $|y|V(y) \in L^{2}(\mathbb{R}^{2})$, $P \in L^{2}(\mathbb{R}^{2})$, and for all vector fields $\varphi \in H^{1}(\mathbb{R}^{2})$ with $|y|\varphi(y)\in L^2(\mathbb{R}^2)$, the following identity holds:
			\begin{equation}\label{eq:weak}
				\begin{aligned}
					&\int_{\mathbb{R}^{2}} \nabla V : \nabla \varphi \, dy - \frac{1}{2} \int_{\mathbb{R}^{2}} (y \cdot \nabla V + V) \cdot \varphi \, dy \\
					=& \int_{\mathbb{R}^{2}} P \operatorname{div} \varphi \, dy + \int_{\mathbb{R}^{2}} V\cdot \nabla \varphi \cdot V \, dy + \int_{\mathbb{R}^{2}} f \cdot \varphi \, dy.
				\end{aligned} \tag{3.45}
			\end{equation}
			Then there exists a universal positive constant $C$ such that
			\begin{equation*}
				\|V\|_{H^{2}(\mathbb{R}^{2})} \leq C \left( \| V\|_{H^{1}(\mathbb{R}^{2})} + \|\nabla V\|_{L^{2}(\mathbb{R}^{2})}^{2} \| V\|_{L^{2}(\mathbb{R}^{2})}^{\frac{1}{2}} + \|f\|_{L^{2}(\mathbb{R}^{2})} \right).
			\end{equation*}
		\end{lemma} 
		
		\begin{proof}
			Let $D_{k}^{h}u$ denote the difference quotient of $u$:
			\begin{equation*}
				D_{k}^{h}u(y) = \frac{u(y + h\mathbf{e}_{k}) - u(y)}{h}, \quad h \in \mathbb{R} \setminus \{0\}.
			\end{equation*}
			Choosing $\varphi(y) = -D_{k}^{-h}(D_{k}^{h}V)$ as a test function in \eqref{eq:weak}, we obtain
			\begin{equation*}
				\begin{aligned}
					-\int_{\mathbb{R}^{2}} \nabla V : \nabla D_{k}^{-h}(D_{k}^{h}V) \, dy &= -\int_{\mathbb{R}^{2}} \nabla V : D_{k}^{-h}(D_{k}^{h}\nabla V) \, dy \\
					&= \int_{\mathbb{R}^{2}} |D_{k}^{h}\nabla V|^2 \, dy = \|D_{k}^{h}\nabla V\|_{L^{2}(\mathbb{R}^{2})}^{2}.
				\end{aligned}
			\end{equation*}
			Integration by parts yields
			\begin{equation*}
				\begin{aligned}
					&\quad \frac{1}{2} \int_{\mathbb{R}^{2}} (y \cdot \nabla V + V) \cdot D_{k}^{-h}(D_{k}^{h}V) \, dy \\
					&= -\frac{1}{2} \int_{\mathbb{R}^{2}} D_{k}^{h}(y \cdot \nabla V + V) \cdot D_{k}^{h}V \, dy \\
					&= -\frac{1}{2} \int_{\mathbb{R}^{2}} \left[ (y + h\mathbf{e}_{k}) \cdot \nabla D_{k}^{h}V + (D_{k}^{h}y) \cdot \nabla V + D_{k}^{h}V \right] \cdot D_{k}^{h}V \, dy \\
					&= -\frac{1}{2} \int_{\mathbb{R}^{2}} y \cdot \nabla D_{k}^{h}V \cdot D_{k}^{h}V \, dy - \frac{1}{2} \int_{\mathbb{R}^{2}} \partial_{y_{k}}V \cdot D_{k}^{h}V \, dy - \frac{1}{2} \int_{\mathbb{R}^{2}} |D_{k}^{h}V|^2 \, dy \\
					&= -\frac{1}{2} \int_{\mathbb{R}^{2}} \partial_{y_{k}}V \cdot D_{k}^{h}V \, dy,
				\end{aligned}
			\end{equation*}
			where we used $\int y \cdot \nabla \psi \cdot \psi \, dy = -\int \psi^2 \, dy$ for $\psi = D_k^h V$.
			For the pressure term, integration by parts gives
			\begin{equation*}
				-\int_{\mathbb{R}^{2}} P \operatorname{div} D_{k}^{-h}(D_{k}^{h}V) \, dy = \int_{\mathbb{R}^{2}} D_{k}^{h}P \operatorname{div} D_{k}^{h}V \, dy = 0,
			\end{equation*}
			owing to the divergence-free condition $\operatorname{div} V = 0$.
			Similarly, for the nonlinear term, we have
			\begin{equation*}
				\begin{aligned}
					-\int_{\mathbb{R}^{2}} V\cdot \nabla D_{k}^{-h}(D_{k}^{h}V) \cdot V \, dy &= \int_{\mathbb{R}^{2}} V(y + h\mathbf{e}_{k}) \cdot \nabla D_{k}^{h}V(y) \cdot D_{k}^{h}V(y) \, dy \\
					&\quad + \int_{\mathbb{R}^{2}} D_{k}^{h} V(y) \cdot \nabla D_{k}^{h}V(y) \cdot V(y) \, dy \\
					&= \int_{\mathbb{R}^{2}} D_{k}^{h} V(y) \cdot \nabla D_{k}^{h}V(y) \cdot V(y) \, dy.
				\end{aligned}
			\end{equation*}
			Substituting all above into \eqref{eq:weak} leads to
			\begin{equation}\label{eq:eneridt}
				\begin{aligned}
					\|D_{k}^{h}\nabla V\|_{L^{2}(\mathbb{R}^{2})}^{2} &= -\int_{\mathbb{R}^{2}} f(y) D_{k}^{-h}D_{k}^{h}V(y) \, dy + \frac{1}{2} \int_{\mathbb{R}^{2}} \partial_{y_{k}}V \cdot D_{k}^{h}V \, dy \\
					&\quad + \int_{\mathbb{R}^{2}} D_{k}^{h} V \cdot \nabla D_{k}^{h}V \cdot V \, dy.
				\end{aligned}
			\end{equation}
			It remains to estimate the right-hand side of the above and pass to the limit as $h \to 0$. 
			
			First, we note that
			\begin{equation*}
				\begin{aligned}
					\int_{\mathbb{R}^{2}} D_{k}^{-h}(D_{k}^{h}V) \cdot D_{k}^{-h}(D_{k}^{h}V) \, dy &=
					-\int_{\mathbb{R}^{2}} (D_{k}^{h}V) \cdot D_{k}^{h}D_{k}^{-h}(D_{k}^{h}V) \, dy \\
					&=
					\int_{\mathbb{R}^{2}} D_{k}^{h}(D_{k}^{h}V) \cdot D_{k}^{h}(D_{k}^{h}V) \, dy \\
					&\leq C \|\nabla D_{k}^{h}V\|_{L^{2}(\mathbb{R}^{2})}^{2}.
				\end{aligned}
			\end{equation*}
			This is just
			\begin{equation*}
				\|D_{k}^{-h}D_{k}^{h}V\|_{L^2}^2 \leq C \|\nabla D_{k}^{h}V\|_{L^{2}(\mathbb{R}^{2})}^{2}.
			\end{equation*}
			This allows us to estimate the first term on the right-hand side of \eqref{eq:eneridt} as
			\begin{equation*}
				-\int_{\mathbb{R}^{2}} f D_{k}^{-h}D_{k}^{h}V \, dy \leq C \|f\|_{L^{2}}^2 + \frac{1}{10} \|\nabla D_{k}^{h}V\|_{L^{2}}^2.
			\end{equation*}
			By Hölder's inequality, the second term on the right-hand side of \eqref{eq:eneridt}  satisfies
			\begin{equation*}
				\frac{1}{2} \int_{\mathbb{R}^{2}} \partial_{y_{k}}V \cdot D_{k}^{h}V \, dy \leq \|\partial_{y_{k}}V\|_{L^{2}(\mathbb{R}^{2})} \|D_{k}^{h}V\|_{L^{2}(\mathbb{R}^{2})} \leq C \|\nabla V\|_{L^{2}(\mathbb{R}^{2})}^{2}.
			\end{equation*}
			Using the Gagliardo–Nirenberg  interpolation inequality, the last term on the right-hand side of \eqref{eq:eneridt} is bounded by
			\begin{equation*}
				\begin{aligned}
					\int_{\mathbb{R}^{2}} D_{k}^{h}V(y) \cdot \nabla D_{k}^{h} V(y) \cdot V(y) \, dy &\leq \|V\|_{L^{6}(\mathbb{R}^{2})} \|D_{k}^{h}V\|_{L^{3}(\mathbb{R}^{2})} \|D_{k}^{h}\nabla V\|_{L^{2}(\mathbb{R}^{2})} \\
					&\leq C \|V\|_{L^{6}(\mathbb{R}^{2})} \|D_{k}^{h}V\|_{L^{2}(\mathbb{R}^{2})}^{2/3} \|D_{k}^{h}\nabla V\|_{L^{2}(\mathbb{R}^{2})}^{4/3} \\
					&\leq C \|V\|_{L^{2}(\mathbb{R}^{2})}^{1/3} \|\nabla V\|_{L^{2}(\mathbb{R}^{2})}^{2/3} \|D_{k}^{h}V\|_{L^{2}(\mathbb{R}^{2})}^{2/3} \|D_{k}^{h}\nabla V\|_{L^{2}(\mathbb{R}^{2})}^{4/3} \\
					&\leq C \|V\|_{L^{2}(\mathbb{R}^{2})}^{1/3} \|\nabla V\|_{L^{2}(\mathbb{R}^{2})}^{4/3}  \|D_{k}^{h}\nabla V\|_{L^{2}(\mathbb{R}^{2})}^{4/3} \\
					&\leq C \|\nabla V\|_{L^{2}(\mathbb{R}^{2})}^{4} \| V\|_{L^{2}(\mathbb{R}^{2})} + \frac{1}{10} \|D_{k}^{h}\nabla V\|_{L^{2}(\mathbb{R}^{2})}^{2}.
				\end{aligned}
			\end{equation*}
			Collecting all estimates above yields
			\begin{equation*}
				\|D_{k}^{h}\nabla V\|_{L^{2}(\mathbb{R}^{2})}^{2} \leq C \|\nabla V\|_{L^{2}(\mathbb{R}^{2})}^2 + C \|\nabla V\|_{L^{2}(\mathbb{R}^{2})}^{4} \| V\|_{L^{2}(\mathbb{R}^{2})} + C \|f\|_{L^{2}(\mathbb{R}^{2})}^{2}.
			\end{equation*}
			Taking $h \to 0$ in the above inequality, we are led to
			\begin{equation*}
				\|\nabla^{2}V\|_{L^{2}(\mathbb{R}^{2})}^{2}  \leq C \|\nabla V\|_{L^{2}(\mathbb{R}^{2})}^2 + C \|\nabla V\|_{L^{2}(\mathbb{R}^{2})}^{4} \| V\|_{L^{2}(\mathbb{R}^{2})} + C \|f\|_{L^{2}(\mathbb{R}^{2})}^{2}.
			\end{equation*}
			Hence,
			\begin{equation*}
				\begin{aligned}
					\|V\|_{H^{2}(\mathbb{R}^{2})}  &= 	\|V\|_{H^{1}(\mathbb{R}^{2})}  +  \|\nabla^{2}V\|_{L^{2}(\mathbb{R}^{2})}\\
					&\leq C \left(  \| V\|_{H^{1}(\mathbb{R}^{2})}+  \|\nabla V\|_{L^{2}(\mathbb{R}^{2})}^{2} \| V\|_{L^{2}(\mathbb{R}^{2})}^{\frac{1}{2}} +  \|f\|_{L^{2}(\mathbb{R}^{2})} \right).
				\end{aligned}
			\end{equation*}
			This completes the proof of the lemma. 
		\end{proof}

		\begin{proposition}\label{pro:H2est}
			Under the same assumption of Proposition \ref{eq:H1est}, 
			we have 
			\begin{equation*}
				\|v_{\text{re}} \|_{H^2(\mathbb{R}^2)}\leq C(A,\beta).
			 \end{equation*}
		\end{proposition}
		
		\begin{proof}
			By Proposition \ref{prop:weighted_H1est}, we have that 
			$|y|v_{\text{re}} \in H^1(\mathbb{R}^2)$.
			We check that the force
			\begin{equation*}
				f  =-(v_0\cdot \nabla v_{\text{re}} + v_{\text{re}}\cdot \nabla  v_0 + v_0\cdot\nabla  v_0) \in L^2(\mathbb{R}^2), ~ \textrm{and}~\|f\|_{L^2(\mathbb{R}^2)}\leq C(A,\beta)
			\end{equation*}
			due to $v_0$ is smooth bounded and satisfies the decay estimates in  Lemma \ref{lem:lineardecay}, and that $v_{re}\in H^1(\mathbb{R}^2)$.
			Applying Lemma \ref{eq:H1est} 
			with this $f$,  
			we obtain   that 
            \begin{equation*}
				\|v_{\text{re}}(y) \|_{H^2(\mathbb{R}^2)}\leq C(A,\beta).
			\end{equation*}
            This finishes the proof of the proposition.
            \end{proof}

		With these regularity estimates in hand, we are going to show the $H^2$-estimate for $\sqrt{1+|y|^2}v_{\text{re}}$, which implies that $\sqrt{1+|y|^2}v_{\text{re}}(y)$ is bounded.
		
		\begin{proposition}[Weighted $H^2$-estimates] \label{propH2est}
            Under the same assumption of Proposition \ref{eq:H1est}, 
			we have 
			\begin{equation}
				\| \sqrt{1+|y|^2}v_{\text{re}}\|_{H^2(\mathbb{R}^2)}\leq C(A,\beta).
			\end{equation}
		\end{proposition}
		
		\begin{proof}
			Having established the $H^1$-estimate of $\sqrt{1+|y|^2}v_{\text{re}}$, we now focus on the second-order derivative estimates. First, we observe the following expansion for $i,j=1,2$:
			\begin{equation*}
				\begin{aligned}
					\partial_{ij}(\sqrt{1+|y|^2}v_{\text{re}}) = & \sqrt{1+|y|^2} \partial_{ij} v_{\text{re}}
					+ \frac{y_i \partial_j v_{\text{re}} + y_j \partial_i v_{\text{re}}}{\sqrt{1+|y|^2}} \
					+ \left( \frac{\delta_{ij}}{\sqrt{1+|y|^2}} - \frac{y_i y_j}{(1+|y|^2)^{3/2}} \right) v_{\text{re}}.
				\end{aligned}
			\end{equation*}
			To obtain $L^2$-estimates for $\partial_{ij}(\sqrt{1+|y|^2}v_{\text{re}})$, it is sufficient to bound $\sqrt{1+|y|^2} \partial_{ij} v_{\text{re}}$ in $L^2$ as all the other terms are bounded in  $L^2$ due to Proposition \ref{pro:H2est}. Furthermore, also in view of the $H^2$-estimate for $v_{\text{re}}$ provided in Proposition \ref{pro:H2est}, it is equivalent to control the term $|y|\nabla^2 v_{\text{re}}$ in $L^2(\mathbb{R}^2)$.

            For the sake of clarity for presentation, we perform the estimates using the  weight $h(y) = |y|$ and the test function $\varphi(y) = -D_k^{-h} (|y|^2 D_k^h v_{\text{re}})$. The more rigorous proof employing  $h_\epsilon(y) = \frac{|y|}{(1+\epsilon|y|^2)^{1/2}}$ and passing $\epsilon \to 0$ follows analogously with minor modifications.
            
                 Below, we use $C$ to denote a constant depending on $A$ and $\beta$, which may vary from line to line.
			Testing \eqref{eq:differencereca} with $\varphi(y) = -D_k^{-h} (|y|^2 D_k^h v_{\text{re}})$ ($k=1,2$), we obtain
			\begin{equation}\label{eq:H2test}
				\begin{aligned}
					&\int_{\mathbb{R}^2} \nabla v_{\text{re}} : \nabla (-D_k^{-h} (|y|^2 D_k^h v_{\text{re}}) ) \, dy 
					- \frac{1}{2} \int_{\mathbb{R}^2} v_{\text{re}} \cdot (-D_k^{-h} (|y|^2 D_k^h v_{\text{re}})) \, dy \\
					& - \frac{1}{2} \int_{\mathbb{R}^2} y \cdot \nabla v_{\text{re}} \cdot (-D_k^{-h} (|y|^2 D_k^h v_{\text{re}}) ) \, dy - \int_{\mathbb{R}^2} q \operatorname{div} (-D_k^{-h} (|y|^2 D_k^h v_{\text{re}} )) \, dy \\
					&= \int_{\mathbb{R}^2} v_{\text{re}} \cdot \nabla (-D_k^{-h} (|y|^2 D_k^h v_{\text{re}}) ) \cdot v_{\text{re}} \, dy - \int_{\mathbb{R}^2} v_0 \cdot \nabla v_{\text{re}} \cdot (-D_k^{-h} (|y|^2 D_k^h v_{\text{re}})) \, dy \\
					& - \int_{\mathbb{R}^2} v_{\text{re}} \cdot \nabla v_0 \cdot (-D_k^{-h} (|y|^2 D_k^h v_{\text{re}}) ) \, dy - \int_{\mathbb{R}^2} v_0 \cdot \nabla v_0 \cdot (-D_k^{-h} (|y|^2 D_k^h v_{\text{re}})) \, dy.
				\end{aligned}
			\end{equation}
            
			We now calculate the above by utilizing the properties of difference quotients and the product rule for the gradient operator. 
			Integration by parts leads to
			\begin{equation}
				\begin{aligned}
					&\quad \int_{\mathbb{R}^2} \nabla v_{\text{re}} : \nabla (-D_k^{-h} (|y|^2 D_k^h v_{\text{re}}) ) dy \\
					&= \int_{\mathbb{R}^2} D_k^h \nabla v_{\text{re}} : \nabla (|y|^2 D_k^h v_{\text{re}}) dy \\
					&= \int_{\mathbb{R}^2} |y|^2 |D_k^h \nabla v_{\text{re}}|^2 dy + 2 \int_{\mathbb{R}^2} (D_k^h \nabla v_{\text{re}}) : (y \otimes D_k^h v_{\text{re}}) dy \\
					&= \||y| D_k^h \nabla v_{\text{re}}\|^2_{L^2(\mathbb{R}^2)} + 2 \int_{\mathbb{R}^2} |y| D_k^h \nabla v_{\text{re}} : \left( \frac{y}{|y|} \otimes D_k^h v_{\text{re}} \right) dy\\
					&	\geq  \frac{9}{10} \||y| \nabla D_k^h v_{\text{re}}\|_{L^2} ^2 - C,
				\end{aligned}
			\end{equation}
			where the second term in above is estimated via the H\"{o}lder inequality and the Young's inequality as follows:
			\begin{equation*}
				\begin{aligned}
					 - 2 \int_{\mathbb{R}^2} |y| D_k^h \nabla v_{\text{re}} : \left( \frac{y}{|y|} \otimes D_k^h v_{\text{re}} \right) dy 
					\leq & C \|\nabla v_{\text{re}}\|_{L^2(\mathbb{R}^2)} \||y| D_k^h \nabla v_{\text{re}}\|_{L^2(\mathbb{R}^2)} \\
					\leq&	 C \|\nabla v_{\text{re}}\|_{L^2(\mathbb{R}^2)}^2 + \frac{1}{10} \||y| D_k^h \nabla v_{\text{re}}\|_{L^2(\mathbb{R}^2)}^2\\
					\leq &C + \frac{1}{10} \||y| \nabla D_k^h v_{\text{re}}\|_{L^2} ^2.
				\end{aligned}
			\end{equation*}
			Utilizing the divergence-free condition $\mathrm{div} \, v_{\text{re}} = 0$, the second and third terms on the left-hand side of  \eqref{eq:H2test} are estimated as:
			\begin{equation*}
				\begin{aligned}
					& - \frac{1}{2} \int_{\mathbb{R}^2} y \cdot \nabla v_{\text{re}} \cdot (-D_k^{-h} (|y|^2 D_k^h v_{\text{re}})) \, dy - \frac{1}{2} \int_{\mathbb{R}^2} v_{\text{re}} \cdot (-D_k^{-h} (|y|^2 D_k^h v_{\text{re}})) \, dy \\
					= & - \frac{1}{2} \int_{\mathbb{R}^2} (y+h\mathbf{e}_k) \cdot \nabla D_k^h v_{\text{re}} |y|^2 D_k^h v_{\text{re}} \, dy - \frac{1}{2} \int_{\mathbb{R}^2} \partial_{y_k} v_{\text{re}} \cdot (|y|^2 D_k^h v_{\text{re}}) \, dy \\
					& - \frac{1}{2} \int_{\mathbb{R}^2} |y|^2 D_k^h v_{\text{re}} D_k^h v_{\text{re}} \, dy \\
					= & \frac{1}{2} \int_{\mathbb{R}^2} |y|^2 D_k^h v_{\text{re}} \cdot D_k^h v_{\text{re}} \, dy - \frac{h}{2} \int_{\mathbb{R}^2} \partial_{y_k}  D_k^h v_{\text{re}}  |y|^2 D_k^h v_{\text{re}} \, dy \\
					& - \frac{1}{2} \int_{\mathbb{R}^2} \partial_{y_k} v_{\text{re}} \cdot (|y|^2 D_k^h v_{\text{re}}) \, dy \\ 
					= & \frac{1}{2} \int_{\mathbb{R}^2} |y|^2 D_k^h v_{\text{re}} \cdot D_k^h v_{\text{re}} \, dy + \frac{h}{2} \int_{\mathbb{R}^2}y_k   D_k^h v_{\text{re}} \cdot  D_k^h v_{\text{re}} \, dy \\
					& - \frac{1}{2} \int_{\mathbb{R}^2} \partial_{y_k} v_{\text{re}} \cdot (|y|^2 D_k^h v_{\text{re}}) \, dy \\ 
                    \geq & \frac{1}{4} \int_{\mathbb{R}^2} |y|^2 D_k^h v_{\text{re}} \cdot D_k^h v_{\text{re}} \, dy + \frac{h}{2} \int_{\mathbb{R}^2}y_k   D_k^h v_{\text{re}} \cdot  D_k^h v_{\text{re}} \, dy \\
					& -  \int_{\mathbb{R}^2} |y|^2|\partial_{y_k} v_{\text{re}}|^2  \, dy \\ 
					\geq& \frac{1}{4} \int_{\mathbb{R}^2} |y|^2 |D_k^h v_{\text{re}}|^2  \, dy + \frac{h}{2} \int_{\mathbb{R}^2}y_k   |D_k^h v_{\text{re}}|^2   \, dy -C.
				\end{aligned}
			\end{equation*}			
			
			For the pressure term, a straightforward calculation yields
			\begin{equation*}
				\begin{aligned}
					\int_{\mathbb{R}^2} q \, \mathrm{div} (-D_k^{-h} (|y|^2 D_k^h v_{\text{re}})) \, dy 
					&= \int_{\mathbb{R}^2} D_k^h q \, \mathrm{div} (|y|^2 D_k^h v_{\text{re}}) \, dy \\
					&= \int_{\mathbb{R}^2} D_k^h q \nabla |y|^2 \cdot D_k^h v_{\text{re}} \, dy \\
					&= 2 \int_{\mathbb{R}^2} D_k^h q y \cdot (D_k^h v_{\text{re}}) \, dy \\
					&\leq C\|\nabla q\|_{L^2}^2 + \frac{1}{20}\|yD_k^h v_{\text{re}}\|_{L^2}^2\\
                    &\leq C + \frac{1}{20}\|yD_k^h v_{\text{re}}\|_{L^2}^2.
				\end{aligned}
			\end{equation*}
			
			Regarding the convection terms on the right-hand side of \eqref{eq:H2test}, integration by parts and using Proposition \ref{pro:H2est} gives
			\begin{equation*} 
				\begin{aligned} & \int_{\mathbb{R}^2} v_{\text{re}} \cdot \nabla (-D_k^{-h} |y|^2 D_k^h v_{\text{re}}) \cdot v_{\text{re}} dy \\
					= & \int_{\mathbb{R}^2} v_{\text{re}}(y+h\mathbf{e}_k) \cdot \nabla (|y|^2 D_k^h v_{\text{re}} ) \cdot D_k^h v_{\text{re}} dy + \int_{\mathbb{R}^2} D_k^h v_{\text{re}} \cdot \nabla (|y|^2 D_k^h v_{\text{re}}) \cdot v_{\text{re}}  dy \\ = & \int_{\mathbb{R}^2} v_{\text{re}}(y+h\mathbf{e}_k) \cdot 2y   ( D_k^h v_{\text{re}}) \cdot D_k^h v_{\text{re}}  dy + \int_{\mathbb{R}^2} |y|^2 v_{\text{re}} (y+h\mathbf{e}_k) \cdot \nabla ( D_k^h v_{\text{re}}) \cdot  D_k^h v_{\text{re}} dy  \\
					& + \int_{\mathbb{R}^2}  D_k^h v_{\text{re}} \cdot 2y   ( D_k^h v_{\text{re}}) \cdot  v_{\text{re}} dy
					+ \int_{\mathbb{R}^2} |y|^2 D_k^h v_{\text{re}} \cdot \nabla ( D_k^h v_{\text{re}}) \cdot v_{\text{re}} dy\\
					\leq& 4\|v_{\text{re}}\|_{L^{\infty}}\||y|D_k^h v_{\text{re}}\|_{L^2} \|D_k^h v_{\text{re}}\|_{L^2}  + 2\|v_{\text{re}}\|_{L^{\infty}}\||y|D_k^h v_{\text{re}}\|_{L^2} \||y| \nabla D_k^h v_{\text{re}}\|_{L^2} \\			
					\leq& C\||y|D_k^h v_{\text{re}}\|_{L^2}^2  +  \frac{1}{20} \||y| \nabla D_k^h v_{\text{re}}\|_{L^2} ^2.
				\end{aligned}
			\end{equation*}
			Similarly, for the interaction terms involving $v_0$, we have
			\begin{equation*}
				\begin{aligned}
					&-\int_{\mathbb{R}^2} v_0 \cdot \nabla v_{\text{re}}  \cdot (-D_k^{-h} (|y|^2 D_k^h v_{\text{re}})) \\
					=& -\int_{\mathbb{R}^2} v_0(y+h\mathbf{e}_k) \cdot D_k^h \nabla v_{\text{re}}  \cdot (|y|^2 D_k^h v_{\text{re}}) -\int_{\mathbb{R}^2}D_k^{h} v_0 \cdot \nabla v_{\text{re}}  \cdot ( |y|^2 D_k^h v_{\text{re}}) \\
					\leq & \||y|v_0\|_{L^{\infty}}\|D_k^h v_{\text{re}}\|_{L^2} \||y| D_k^h \nabla v_{\text{re}}\|_{L^2}  +\||y|^2 \nabla v_0\|_{L^{\infty}} \|\nabla v_{\text{re}}\|_{L^2}^2  \\
					\leq & C \||y|v_0\|_{L^{\infty}}^2 \|\nabla v_{\text{re}}\|_{L^2}^2  + \||y|^2 \nabla v_0\|_{L^{\infty}} \|\nabla v_{\text{re}}\|_{L^2}^2 + \frac{1}{20}\||y| D_k^h \nabla v_{\text{re}}\|_{L^2}^2\\
					\leq & C  + \frac{1}{10}\||y| D_k^h \nabla v_{\text{re}}\|_{L^2}^2.
				\end{aligned}
			\end{equation*}
			For the last two terms on the right-hand side of \eqref{eq:H2test}, we have
			\begin{equation*}
				\begin{aligned}
					& -\int_{\mathbb{R}^2} v_{\text{re}} \cdot \nabla v_0 \cdot (-D_k^{-h} (|y|^2 D_k^h v_{\text{re}})) dy \\
					= & -\int_{\mathbb{R}^2} v_{\text{re}}(y + he_k) \cdot D_k^h \nabla v_0 \cdot (|y|^2 D_k^h v_{\text{re}}) dy - \int_{\mathbb{R}^2} D_k^h v_{\text{re}} \cdot \nabla v_0 \cdot (|y|^2 D_k^h v_{\text{re}}) dy \\
					\leq & C\| | y |^2 \nabla^2 v_0 \|_{L^{\infty}}\|v_{\text{re}}\|_{L^2} \|\nabla v_{\text{re}}\|_{L^2} + C\| | y |^2 \nabla v_0 \|_{L^{\infty}}\|\nabla v_{\text{re}}\|_{L^2}^2\leq C.
				\end{aligned}
			\end{equation*}
			And		
			\begin{equation*}
				\begin{aligned}
					& -\int_{\mathbb{R}^2} v_0 \cdot \nabla v_0 \cdot (-D_k^{-h} (|y|^2 D_k^h v_{\text{re}})) dy \\
					= & -\int_{\mathbb{R}^2} v_0(y + he_k) \cdot D_k^h \nabla v_0 \cdot (|y|^2 D_k^h v_{\text{re}}) dy - \int_{\mathbb{R}^2} D_k^h v_0 \cdot \nabla v_0 \cdot (|y|^2 D_k^h v_{\text{re}}) dy \\
					\leq & \| | y | v_0\|_{L^{\infty}} \| \nabla^2 v_0 \|_{L^2} \| y \|D_k^h v_{\text{re}}\|_{L^2} + \| | y |^2 \nabla v_0 \|_{L^{\infty}} \|\nabla v_0\|_{L^2} \|\nabla v_{\text{re}}\|_{L^2}\leq C.
				\end{aligned}
			\end{equation*}
			Collecting all estimates and taking $h \to 0$ we obtain that
			\begin{equation}
				\||y|\nabla^2 v_{\text{re}}\|^2_{L^2(\mathbb{R}^2)} \leq C.
			\end{equation}
			This concludes the proof.
		\end{proof}
		
		\subsection{Weighted $H^2$-Estimates of $\nabla v_{\textnormal{re}}$}
        We continue to establish the  $H^2$-estimates for $\nabla v_{\textnormal{re}}$.
		The main goal of this section is to prove the following.
		\begin{proposition} \label{pro:H3}
			Under the same assumption of Proposition \ref{eq:H1est}, we have 
			\begin{equation*}
				\|v_{\text{re}} \|_{H^3(\mathbb{R}^2)}\leq C(A,\beta), ~\|\sqrt{1+|y|^2}\nabla v_{\text{re}} \|_{H^2(\mathbb{R}^2)}\leq C(A,\beta).
			\end{equation*}
		\end{proposition}
		First, we give a lemma considering $H^2$-estimates for  the following general linear system, which is quite similar to Lemma \ref{eq:H1est}.
		\begin{lemma}[$H^2$-estimates]\label{eq:H2est}
			Let $f(y) \in L^{2}(\mathbb{R}^{2})$ and the divergence-free vector field $\overline{V}\in H^2(\mathbb{R}^2)$ be given. Assume that $(V, P)$ is a weak solution to the following system:
			\begin{equation*}
            \left\{
				\begin{aligned}
				  &  -\Delta V + \overline{V} \cdot \nabla V - \dfrac{1}{2}(y \cdot \nabla V + V) + \nabla P = f , \\
					&\div V = 0,
				\end{aligned}
                \right. \text{in } \mathbb{R}^{2}.
			\end{equation*}
			Specifically, $V \in H^1(\mathbb{R}^{2})$, $|y|V(y) \in L^{2}(\mathbb{R}^{2})$, $P \in L^{2}(\mathbb{R}^{2})$, and for all vector fields $\varphi \in H^{1}(\mathbb{R}^{2})$, the following identity holds:
			\begin{equation}\label{eq:weak1}
				\begin{aligned}
					&\int_{\mathbb{R}^{2}} \nabla V : \nabla \varphi \, dy - \frac{1}{2} \int_{\mathbb{R}^{2}} (y \cdot \nabla V + V) \cdot \varphi \, dy \\
					&= \int_{\mathbb{R}^{2}} P \operatorname{div} \varphi \, dy + \int_{\mathbb{R}^{2}} \overline{V}\cdot \nabla \varphi \cdot V \, dy + \int_{\mathbb{R}^{2}} f \cdot \varphi \, dy.
				\end{aligned} 
			\end{equation}
			Then, there exists a  positive constant $C$ such that
			\begin{equation*}
				\|V\|_{H^{2}(\mathbb{R}^{2})} \leq C \left( \| \overline{V}\|_{H^{2}(\mathbb{R}^{2})}, \| V\|_{H^{1}(\mathbb{R}^{2})} ,\|f\|_{L^{2}(\mathbb{R}^{2})} \right).
			\end{equation*}
		\end{lemma} 
		\begin{proof}
        Similar to the proof of   Lemma \ref{eq:H1est}, we choose $\varphi=- D_{k}^{-h}D_{k}^{h}V$.  All the other terms are identical to those in
         Lemma \ref{eq:H1est}, 
			we only need to deal with the following term. Using $\div \overline{V}=0$, we have that
			\begin{equation*}
				\begin{aligned}
					-\int_{\mathbb{R}^{2}} \overline{V}\cdot \nabla D_{k}^{-h}(D_{k}^{h}V) \cdot V \, dy &= \int_{\mathbb{R}^{2}} \overline{V}(y + h\mathbf{e}_{k}) \cdot \nabla D_{k}^{h}V(y) \cdot D_{k}^{h}V(y) \, dy \\
					&\quad + \int_{\mathbb{R}^{2}} D_{k}^{h} \overline{V}(y) \cdot \nabla D_{k}^{h}V(y) \cdot V(y) \, dy \\
					&= \int_{\mathbb{R}^{2}} D_{k}^{h} \overline{V}(y) \cdot \nabla D_{k}^{h}V(y) \cdot V(y) \, dy\\
					&\leq C \|V\|_{L^{6}(\mathbb{R}^{2})} \|D_{k}^{h}\overline{V}\|_{L^{3}(\mathbb{R}^{2})} \|D_{k}^{h}\nabla V\|_{L^{2}(\mathbb{R}^{2})} \\
					&\leq \frac{1}{10}\|D_{k}^{h}\nabla V\|_{L^{2}(\mathbb{R}^{2})}^2  + C \|\overline{V}\|_{H^2(\mathbb{R}^{2})}^2\|V\|_{H^1(\mathbb{R}^{2})}^2.
				\end{aligned}
			\end{equation*}
            We finish the proof.
		\end{proof}
		
		We are now in a position to prove Proposition \ref{pro:H3}.
		\begin{proof}[Proof of Proposition \ref{pro:H3}]
			Let $\mathbb{P} = \textrm{Id} - \nabla (\Delta)^{-1}\div$ be  the Leray-Hopf projection onto the divergence-free
			vector fields.
			We know that $v_{\text{re}}$ solves
			\begin{equation}
				\begin{aligned}
					-\Delta v_{\text{re}} &= \frac{1}{2}y \cdot \nabla v_{\text{re}} + \frac{1}{2}v_{\text{re}} - \mathbb{P}  (v_{\text{re}} \cdot \nabla v_{\text{re}} + v_0 \cdot \nabla v_{\text{re}} + v_{\text{re}} \cdot \nabla v_0 + v_0 \cdot \nabla v_0) \\
					&=f.
				\end{aligned}		
			\end{equation}
			Thanks to Lemma \ref{lem:lineardecay}, Propositions \ref{pro:H2est} and \ref{propH2est}, it is easy to see that this $f\in H^1(\mathbb{R}^2)$. By the classical elliptic regularity theory and the fact that $v_{\textrm{re}}\in H^2(\mathbb{R}^2)$, we immediately obtain
			\begin{equation*}
				\|v_{\textrm{re}}\|_{ H^3(\mathbb{R}^2)} \le C.
			\end{equation*}
			Hence, we also have that 
			\begin{equation*}
				\|\nabla v_{\textrm{re}}\|_{ L^{\infty}(\mathbb{R}^2)} \le C.
			\end{equation*}
			Let  $V_k = \sqrt{1+|y|^2} \partial_{y_k} v_{\text{re}}$ and $q_k = \sqrt{1+|y|^2} \partial_{y_k} q$ for $k=1,2$, we can find that	
\begin{equation}\label{eq:neweq}
\begin{aligned}
    &-\Delta V_k + v_{\text{re}} \cdot \nabla V_k - \frac{1}{2} y \cdot \nabla V_k - \frac{1}{2} V_k + \nabla q_k \\
    =& -\sqrt{1+|y|^2} \partial_{y_k}(v_{\text{re}} \cdot \nabla v_0 + v_0 \cdot \nabla v_{\text{re}} + v_0 \cdot \nabla v_0) - V_k \cdot \nabla v_{\text{re}} \\
    & + v_{\text{re}} \cdot \frac{y}{\sqrt{1+|y|^2}} \partial_{y_k} v_{\text{re}} - \frac{1}{2} V_k - \frac{|y|^2}{2\sqrt{1+|y|^2}} \partial_{y_k} v_{\text{re}} \\
    & - 2 \frac{y}{\sqrt{1+|y|^2}} \cdot \nabla \partial_{y_k} v_{\text{re}} - \frac{2+|y|^2}{(1+|y|^2)^{3/2}} \partial_{y_k} v_{\text{re}} + \frac{y}{\sqrt{1+|y|^2}} \partial_{y_k} q.
\end{aligned}
\end{equation}

			We see that
			\begin{equation*}
				\sqrt{1+|y|^2} \partial_{y_k} (v_{\text{re}} \cdot \nabla v_0 + v_0 \cdot \nabla v_{\text{re}} + v_0 \cdot \nabla v_0) \in L^2(\mathbb{R}^2),
			\end{equation*}
			and also 
			\begin{equation*}
				\|V_k\|_{L^2(\mathbb{R}^2)} \leq \|\sqrt{1+|y|^2} \nabla v_{\text{re}}\|_{L^2(\mathbb{R}^2)},
			\end{equation*}			
			and		
			\begin{equation*}
				\left\| \frac{y}{\sqrt{1+|y|^2}} \cdot \nabla \partial_{y_k} v_{\text{re}} \right\|_{L^2(\mathbb{R}^2)} \leq C \|v_{\text{re}}\|_{H^2(\mathbb{R}^2)}.
			\end{equation*}
            The other terms on the right-hand side of \eqref{eq:neweq} can  be seen belong to $L^2(\mathbb{R}^2)$.
			Applying Lemma \ref{eq:H2est} to \eqref{eq:neweq}, we know that
			\begin{equation*}
				\|\sqrt{1+|y|^2}\nabla v_{\textrm{re}}\|_{H^2(\mathbb{R}^2)} \le C.
			\end{equation*}
			This finishes the proof of the proposition.
		\end{proof}

		\subsection{Pointwise Decay estimates for the self-similar solution}
		The main result of this section is Theorem \ref{thm:optimaldecay}, which asserts the optimal decay rates for the remainder term $v_{\textrm{re}}$. 
		Propositions \ref{propH2est} and \ref{pro:H3} implies the following obvious poinitwise estimate, which will be improved through a kind a bootstrap argument using the estimates for the Stokes system.
		\begin{theorem}[Pointwise decay estimates]\label{thm:pointdecay}
			Under the same assumption of Proposition \ref{eq:H1est}, we have 
			for all $y \in \mathbb{R}^2$ that
			\begin{equation*}
				|\nabla^k v_{\text{re}}(y)| \le\frac{C(A,\beta)}{1+|y|}, ~k =0,1.
			\end{equation*}
		\end{theorem}
		
		\begin{proof}
			From Proposition \ref{propH2est}, we have $\sqrt{1+|y|^2}v_{\text{re}} \in H^2(\mathbb{R}^2)$. By the Sobolev embedding $H^2(\mathbb{R}^2) \hookrightarrow L^\infty(\mathbb{R}^2)$, we have:
			\[
			\| \sqrt{1+|y|^2} v_{\text{re}} \|_{L^\infty} \le C \| \sqrt{1+|y|^2} v_{\text{re}} \|_{H^2} \le C.
			\]
			This implies 
			\begin{equation*}
				|v_{\text{re}}(y)| \le \frac{C}{1+|y|}.
			\end{equation*}
			Similarly, it follows from Proposition \ref{pro:H3} and Sobolev embedding that
			\begin{equation*}
				|\nabla v_{\text{re}}(y)| \le \frac{C}{1+|y|}.
			\end{equation*}
		\end{proof}
		
		We are going to improve the decay rate in Theorem \ref{thm:pointdecay} using the estimates for the Stokes system. 
		We first study a nonhomogeneous Stokes system with a singular
		force. 
		\begin{equation}	\label{eq:nonhomogeneous-Stokes}
			\left\{
			\begin{aligned}
				\partial_t \vartheta - \Delta \vartheta + \nabla \pi &= t^{-1}\div_x F\left(\frac{x}{\sqrt{t}}\right), \\
				\div \vartheta &= 0,
			\end{aligned}
			\right. \text{ in } \mathbb{R}^2 \times (0, +\infty).
		\end{equation}
		We have the  following lemma, which plays an important role in improving decay estimates for weak solutions of equation 	\eqref{eq:differencereca}. This kind of lemma already appeared in \cite[Lemma 4.1]{Jia14} in the 3D case; see also \cite[Lemma 3.15]{Yang26}. We present the lemma for the 2D case; the proof is essentially the same.
		\begin{lemma}\label{lem:nonhomogeneous-Stokes}
				If for all $x \in \mathbb{R}^2$, $F(x)$ satisfies $|F(x)| \leq\frac{C}{1+|x|^2}$, then the solution $\vartheta(x,t)$ to \eqref{eq:nonhomogeneous-Stokes} is given by
				\begin{equation}\label{eq:oseentensor}
					\vartheta(x,t) = \int_0^t   e^{(t-s)\Delta} \mathbb{P} \div_x s^{-1} F\left(\frac{\cdot}{\sqrt{s}}\right) ds.
				\end{equation}
				Let $\Theta(x) = \vartheta(x,1)$, then
				\begin{equation}
					|\Theta(x)| \leq \frac{C}{1+|x|^2}.
				\end{equation}
				Moreover, if $F \in C^{0,1}(\mathbb{R}^2)$ satisfies $\div_x F(\frac{x}{\sqrt{t}}) = t^{-1/2} f(\frac{x}{\sqrt{t}})$ with $|f(x)| \leq \frac{C}{1+|x|^2}$, 
				we have
				\begin{equation}
					|\nabla \Theta (x)| \leq \frac{C}{1+|x|^2}.
				\end{equation}
		\end{lemma}
		\begin{proof}
			It is easy to check $\vartheta$ can be represented as \eqref{eq:oseentensor}. The kernel of $e^{t\Delta}\mathbb{P}$, usually called the Oseen kernel, which we denote by $S(x,t)$ is known to have the following estimates (\cite{BrandoleseJMPA07}):
			\begin{equation*}
				|\partial_t^l\nabla^k_x S(x,t)|\leq C(k,l) t^{-l} (|x|+\sqrt{t})^{-2-k}, ~ k, l\in \mathbb{N}^+.
			\end{equation*}
			Note that
			\begin{equation*}
				\begin{aligned}
					\Theta(x)= \vartheta(x,1) &= \int_0^1  e^{(1-s)\Delta} \mathbb{P} \div_x s^{-1} F\left(\frac{\cdot}{\sqrt{s}}\right) ds\\
					&= 	-\int_0^1 \int_{\mathbb{R}^2}\nabla S(x-y,1-s) s^{-1} F\left(\frac{y}{\sqrt{s}}\right) dy ds.
				\end{aligned}
			\end{equation*}
			Hence, based on the pointwise estimate of the Oseen kernel, we have
			\begin{equation}\label{eq:estTheta}
				\begin{aligned}
					|	\Theta(x)|\leq C \int_0^1 \int_{\mathbb{R}^2}\frac{1}{(\sqrt{1-s}+|x-y|)^3} \frac{1}{(\sqrt{s}+|y|)^2} dy ds.
				\end{aligned}
			\end{equation}
			If $F \in C^1(\mathbb{R}^2)$ satisfies $\div_x F(\frac{x}{\sqrt{t}}) = t^{-1/2} f(\frac{x}{\sqrt{t}})$ with $|f(x)| \leq C(1 + |x|)^{-2}$, 
			we have
			\begin{equation*}
				\begin{aligned}
					\Theta(x)= \vartheta(x,1) &= \int_0^1  e^{(1-s)\Delta} \mathbb{P} s^{-\frac{3}{2}} f\left(\frac{\cdot}{\sqrt{s}}\right) ds\\
					&= 	\int_0^1 \int_{\mathbb{R}^2} S(x-y,1-s)  s^{-\frac{3}{2}} f\left(\frac{y}{\sqrt{s}}\right)  dy ds.
				\end{aligned}
			\end{equation*}
			Then 
			\begin{equation*}
				\begin{aligned}
					\nabla \Theta(x)=  
						\int_0^1 \int_{\mathbb{R}^2} \nabla S(x-y,1-s)  s^{-\frac{3}{2}} f\left(\frac{y}{\sqrt{s}}\right) dy  ds.
				\end{aligned}
			\end{equation*}
			Therefore, we obtain, based on the pointwise estimate of the Oseen kernel, that 
			\begin{equation}\label{eq:estnablaThe}
				\begin{aligned}
					|\nabla \Theta(x)| & \leq 	\int_0^1 s^{-\frac{1}{2}} \int_{\mathbb{R}^2} \frac{1}{(\sqrt{1-s}+|x-y|)^3} \frac{1}{(\sqrt{s}+|y|)^2}dy ds.
				\end{aligned}
			\end{equation}
			We now prove that 
			\begin{equation*}
				|\nabla \Theta(x)| \leq \frac{C}{1+|x|^2},
			\end{equation*}
			by using \eqref{eq:estnablaThe}.
			The proof of $| \Theta(x)| \leq \frac{C}{1+|x|^2}$ using \eqref{eq:estTheta} is similar and indeed simpler and thus omitted.
			We have
			\begin{equation*}
				\begin{aligned}
					|\nabla \Theta(x)| & \leq 	\int_0^1 s^{-\frac{1}{2}} \int_{\mathbb{R}^2} \frac{1}{(\sqrt{1-s}+|x-y|)^3} \frac{1}{(\sqrt{s}+|y|)^2}dy ds\\
					& =  	\int_0^1 s^{-\frac{1}{2}} \int_{|y|\leq \frac{|x|}{2}}   
					+ \int_{\frac{|x|}{2}\leq |y|\leq 2|x|}  + 
					\int_{|y|\geq 2|x| }  \frac{1}{(\sqrt{1-s}+|x-y|)^3} \frac{1}{(\sqrt{s}+|y|)^2}dy ds\\
					&:=J_1+J_2+J_3. 
				\end{aligned}
			\end{equation*}
			We only consider $|x|\ge 10$, the boundedness when $|x|\leq 10$ is clear. Then 
			\begin{equation*}
				J_1 \leq \int_0^1 s^{-\frac{1}{2}} \int_{|y|\leq \frac{|x|}{2}}   \frac{8}{|x|^3} \frac{1}{(\sqrt{s}+|y|)^2} dy ds
				\leq C|x|^{-3} \ln(1+|x|).
			\end{equation*}
			\begin{equation*}
				\begin{aligned}
					J_2 &\leq \int_0^1 s^{-\frac{1}{2}} \int_{\frac{|x|}{2}\leq |y|\leq 2|x|} \frac{1}{(\sqrt{1-s}+|x-y|)^3}  \frac{1}{(\sqrt{s}+|y|)^2} dy ds
					\\
					&\leq C\frac{1}{|x|^2}\int_0^1 s^{-\frac{1}{2}} \int_{\frac{|x|}{2}\leq |y|\leq 2|x|} \frac{1}{(\sqrt{1-s}+|x-y|)^3}   dy ds
					\\
					&\leq C\frac{1}{|x|^2}\int_0^1 s^{-\frac{1}{2}} \int_{|z|\leq 3|x|} \frac{1}{(\sqrt{1-s}+|z|)^3}   dz ds\\
					&
					\leq C\frac{1}{|x|^2}\int_0^1 s^{-\frac{1}{2}}(1-s)^{-\frac{1}{2}} ds \leq C\frac{1}{|x|^2}.
				\end{aligned}
			\end{equation*}
			Similarly, one can prove
			\begin{equation*}
				\begin{aligned}
					J_3 &\leq \int_0^1 s^{-\frac{1}{2}} \int_{ |y|\geq 2|x|} \frac{1}{(\sqrt{1-s}+|x-y|)^3}  \frac{1}{(\sqrt{s}+|y|)^2} dy ds
					\\
					&\leq C\frac{1}{|x|^2}\int_0^1 s^{-\frac{1}{2}} \int_{|y|\geq 2|x|} \frac{1}{(\sqrt{1-s}+|y|)^3}   dy ds
					\\
					&
					\leq C\frac{1}{|x|^2}\int_0^1 s^{-\frac{1}{2}}(1-s)^{-\frac{1}{2}} ds \leq C\frac{1}{|x|^2}.
				\end{aligned}
			\end{equation*}
            We finish the proof.
		\end{proof}
		We now give the main estimates for this section.
		\begin{theorem}[Improved pointwise decay estimates]\label{thm:optimaldecay}
			Under the same assumption of Proposition \ref{prop:weighted_H1est}, we have 
			for all $x \in \mathbb{R}^2$ that 
			\begin{equation*}
				| v_{\text{re}}(x)| \le C(A,\beta)(1+|x|)^{-2},
			\end{equation*}
			and that
			\begin{equation*}
				|\nabla v_{\text{re}}(x)| \le C(A,\beta)(1+|x|)^{-(2+\beta)}.
			\end{equation*}
			Moreover, if the initial data $u_0\in C^2(S^1)$, we have the following optimal decay rate estimates
			\begin{equation*}
				| v_{\text{re}}(x)| \le C(\|u_0\|_{C^2(S^1)})(1+|x|)^{-3}\ln(1+|x|).
			\end{equation*}
		\end{theorem}
		
		\begin{proof}
			Let $\vartheta(x,t)$ = $\frac{1}{\sqrt{t}}v_{\textrm{re}}\left(\frac{x}{\sqrt{t}}\right)$. Then this
			$\vartheta(x,t)$ solves
			\begin{equation}	\label{eq:nonhomogeneous-Stokes1}
				\left\{
				\begin{aligned}
					\partial_t \vartheta - \Delta \vartheta + \nabla \pi &= t^{-1}\div_x F\left(\frac{x}{\sqrt{t}}\right), \\
					\div \vartheta &= 0,
				\end{aligned}
				\right. \text{ in } \mathbb{R}^2 \times (0, +\infty).
			\end{equation}
			with 
			\begin{equation*}
				F= (v_{\textrm{re}} + v_0)\otimes(v_{\textrm{re}} + v_0).
			\end{equation*}
			It follows from Lemma \ref{lem:lineardecay} and Theorem \ref{thm:pointdecay} that
			\begin{equation*}
				|F|\leq \frac{C(A,\beta)}{1+|x|^2}.
			\end{equation*}
			Using Lemma \ref{lem:nonhomogeneous-Stokes}, we have
			\begin{equation*}
				|\vartheta(x,1)|= |v_{\textrm{re}}(x)|\leq \frac{C(A,\beta)}{1+|x|^2}.
			\end{equation*}
			Moreover, 
			\begin{equation*}
				t^{-1}\div_x F\left(\frac{x}{\sqrt{t}}\right) = t^{-\frac{3}{2}}f\left(\frac{x}{\sqrt{t}}\right),
			\end{equation*}
			where
			\begin{equation*}
				f= (v_{\textrm{re}} + v_0)\cdot \nabla (v_{\textrm{re}} + v_0).
			\end{equation*}
			It also follows from Lemma \ref{lem:lineardecay} and Theorem \ref{thm:pointdecay} that
			\begin{equation*}
				|f(x, t)|\leq \frac{C(A,\beta)}{1+|x|^2}.
			\end{equation*}
			By then, Lemma \ref{lem:nonhomogeneous-Stokes} gives
			\begin{equation*}
				|\nabla v_{\textrm{re}}(x)|\leq \frac{C(A,\beta)}{1+|x|^2}.
			\end{equation*}
            It follows from Lemma \ref{lem:lineardecay} and the above that
            \begin{equation*}
				|f(x, t)|\leq \frac{C(A,\beta)}{1+|x|^{2+\beta}}.
			\end{equation*}
			We then can improve the decay of  $\nabla v_{\textrm{re}}$ as 
			\begin{equation}\label{eq:decayofgradient}
				\begin{aligned}
					| \nabla v_{\textrm{re}}(x)| & \leq \left|\int_0^1 \int_{\mathbb{R}^2} \nabla S(x-y,1-s)  s^{-\frac{3}{2}} f\left(\frac{y}{\sqrt{s}}\right)  dy ds\right|\\
					&\leq C \int_0^1 \int_{\mathbb{R}^2}\frac{1}{(\sqrt{1-s}+|x-y|)^3} s^{\frac{\beta-1}{2}}\frac{1}{(\sqrt{s}+|y|)^{2+\beta}} dy ds\\
					& \leq \frac{C}{1+|x|^{2+\beta}}, ~ 0<\beta \leq 1.
				\end{aligned}						
			\end{equation}
			If the initial data has higher regularity with $\|u_0\|_{C^2(S^1)}<+\infty$, then
			we can apply \cite[Lemma 4]{Brandolese09} to the vector field $F=v_0\otimes v_0= e^{\Delta}u_0\otimes e^{\Delta}u_0$, to  conclude that
			\begin{equation*}
				\vartheta_1(x,t) = \int_0^t   e^{(t-s)\Delta} \mathbb{P} \div_x s^{-1} (e^{\Delta}u_0\otimes e^{\Delta}u_0)\left(\frac{\cdot}{\sqrt{s}}\right) ds
			\end{equation*}
			at time $t=1$	is bounded by $C(|u_0\|_{C^2(S^1)})(1+|x|)^{-3}\ln(1+|x|)$. The proof  makes essential use of the cancellation properties of the Oseen kernel.
			On the other hand, all the other terms in $(v_{\textrm{re}} + v_0)\otimes(v_{\textrm{re}} + v_0)$ except $v_0\otimes v_0$ has bound $C(1+|x|)^{-3}$ at least. And their contribution in \eqref{eq:nonhomogeneous-Stokes1} can be calculated similarly to that of Lemma \ref{lem:nonhomogeneous-Stokes} and bound by 
			\begin{equation*}
				\begin{aligned}
					C \int_0^1 \int_{\mathbb{R}^2}\frac{1}{(\sqrt{1-s}+|x-y|)^3} \frac{1}{(\sqrt{s}+|y|)^3} dy ds
					\leq \frac{C}{1+|x|^3}.
				\end{aligned}						
			\end{equation*}
            This finishes the proof of the theorem.
		\end{proof}

		\section{Proof of the main theorem}\label{sec:mainthm}
		The results in Sections \ref{sec:energy estimates} and \ref{sec:pointwise} implies the
	 following a priori estimates for the self-similar solutions on 	$\mathbb{R}^2$. 
		\begin{theorem}\label{thm:apriotest}
			Let divergence-free initial data $u_0$ on $\mathbb{R}^2$ be self-similar, 
			and that
			$  \int_{\partial S^1} u_0\cdot n d\sigma = 0 $, 
			\begin{equation*}
				A=\|u_0\|_{C^{0,\beta}(S^1)}<+\infty,
			\end{equation*}
			Assume 
			$u(x,t)$ is a self-similar solution to \eqref{eq:NS} in the sense of Definition \ref{def:energy persol}.
			Then $v(x) =  u(x,1)$
			the solution profile at time t = 1, belongs to $C^{\infty}(\mathbb{R}^2)$.
			Let $v_0 = e^{\Delta}u_0$ and $v_{\textrm{re}}(x) = v(x) -v_0 (x)$, assume $v_{\textrm{re}}\in X$, then we have
			\begin{equation*}
				\left|v_{\textrm{re}}(x)\right|	= \left|v(x)- v_0 (x)\right|\leq \frac{C(A,\beta)}{1+|x|^2}.
			\end{equation*}
			Also, 
			\begin{equation*}
				|\nabla v_{\text{re}}(x)| \le \frac{C(A,\beta)}{(1+|x|)^{2+\beta}}.
			\end{equation*}
		\end{theorem}
           We also  recall that 
        \begin{equation}\label{eq:defofX}
				X=\{ \Bv:\ \div \Bv =0, \ \Bv\in C^1(\mathbb{R}^2), ~ \|\Bv\|_{X}<+\infty\},
			\end{equation}
			and the norm $\|\cdot\|_X$ is 
			\begin{equation*}\label{eq:defofnorm1}
				\|\Bv\|_{X}= \|(1+|x|^2)\Bv\|_{C(\mathbb{R}^2)} +\|(1+|x|^2)\nabla \Bv\|_{C(\mathbb{R}^2)}.
			\end{equation*}
            
		We  can now prove Theorem \ref{thm:main}.
		\begin{proof}
			We try apply the Leray-Schauder
			fixed point theorem.
			For each $\Bv \in X$, we define its self-similar extension by
			\begin{equation*}
				E\Bv (x,t) = t^{-1/2} \Bv(t^{-1/2} x), \quad (x,t) \in \mathbb{R}^3\times (0,\infty).
			\end{equation*}
			
			We now define an operator $K: X \times [0, 1] \to X$ by solving the  nonhomogeneous Stokes system and restricting the solution at time $t=1$:
			\begin{equation} 
				K(\Bv, \sigma) := -\Phi[(\sigma e^{t\Delta}u_0 + E\Bv) \otimes (\sigma e^{t\Delta}u_0 + E\Bv)] \big|_{t=1}.
			\end{equation}
			In the above, $\Phi$ is the solution operator of the nonhomogeneous Stokes system defined by formula \eqref{eq:oseentensor}, in which one takes
			\begin{equation*}
				\left. F= (\sigma e^{t\Delta}u_0 + E\Bv) \otimes (\sigma e^{t\Delta}u_0 + E\Bv)\right|_{t=1}.
			\end{equation*}

			Note that for $\Bv \in X$ with $\|\Bv\|_X < M$ and $0 \leq \sigma \leq 1$, the force 	$\left. F= (\sigma v_0 + E\Bv) \otimes (\sigma v_0 + E\Bv)\right|_{t=1}$ satisfies
			\begin{equation*}
				|F(x)|\leq \frac{C(M, \|v_0\|_X)}{1+|x|^2}\leq \frac{C(A,M)}{1+|x|^2}.
			\end{equation*}
			It follows from Lemma \ref{lem:nonhomogeneous-Stokes} that 
			\begin{equation*}
				|K(\Bv, \sigma) (x,1)|,~ |\nabla K(\Bv, \sigma) (x,1)| \leq \frac{C(A, M)}{1+|x|^2}\leq C(A,M)
			\end{equation*}
			hence, $K(\Bv, \sigma)\in X$ and $\| K(\Bv, \sigma) \|_X \leq C(A,M)$. 
			Thus $K$ indeed maps bounded sets in $X \times [0, 1]$ into bounded sets in $X$.
			
			Furthermore, $K$ is compact because its main term $\Phi[(\sigma e^{t\Delta}u_0 ) \otimes (\sigma e^{t\Delta}u_0 )] \big|_{t=1} = \sigma^2 \Phi[( e^{t\Delta}u_0 ) \otimes ( e^{t\Delta}u_0 )] \big|_{t=1}$ has a one-dimensional range. While the other terms of 
            $(\sigma e^{t\Delta}u_0 + E\Bv) \otimes (\sigma e^{t\Delta}u_0 + E\Bv)$ except $\sigma e^{t\Delta}u_0  \otimes \sigma e^{t\Delta}u_0$ decay with rate faster or equal to $\frac{1}{1+|x|^3}$. Similar to calculations in Lemma \ref{lem:nonhomogeneous-Stokes}, we know that 
            $\Phi[\sigma e^{t\Delta}u_0  \otimes  E\Bv + E\Bv\otimes (\sigma e^{t\Delta}u_0 + E\Bv)] \big|_{t=1}$ decay at infinity at least with the rate $\frac{1}{1+|x|^3}$. 
            Due to the same reason, we see that $\nabla \Phi[\sigma e^{t\Delta}u_0  \otimes  E\Bv + E\Bv\otimes (\sigma e^{t\Delta}u_0 + E\Bv)] \big|_{t=1}$ decay at infinity at least with the rate $\frac{1}{1+|x|^3}$. 
           And they possess at least $C^{1,\alpha}$ ($\alpha>0$) regularity in  any bounded region due to classical elliptic estimates for the linear Stokes operator. It is then not hard to see that $K$ is compact.
			We can also easily verify that
$K$ is continuous   based on similar  estimates to Lemma \ref{lem:nonhomogeneous-Stokes}.

			We now try to solve a fixed-point problem.
			\begin{equation}\label{eq:fixedpoint}
				\Bv = K(\Bv, \sigma) \quad \text{in } X
			\end{equation}
			that satisfies the following:
			\begin{itemize}
				\item[(1)] $K$ is continuous; $K(\cdot, \sigma)$ is compact for each $\sigma$.
				\item[(2)] It is uniquely solvable in $X$ for small $\sigma$ by classical result; see for example \cite{Cannone96} and \cite[Proposition 3]{Brandolese09}.
				\item[(3)] We have a priori estimates in $X$ for solutions $(\Bv, \sigma) \in X \times [0, 1]$ by Theorem \ref{thm:apriotest}.
			\end{itemize}
			
			By the Leray-Schauder theorem, there is a solution $v_{\textrm{re}} \in X$ of \eqref{eq:fixedpoint} with $\sigma = 1$. It follows that $ u_{\textrm{re}} =  E v_{\textrm{re}}$ satisfies the nonhomogeneous Stokes system \eqref{eq:nonhomogeneous-Stokes}  with $f = -(e^{t\Delta}u_0 + u_{\textrm{re}}) \otimes (e^{t\Delta}u_0 + u_{\textrm{re}})$, and hence $u = e^{t\Delta} u_0 + u_{\textrm{re}}$ is a self-similar  solution of \eqref{eq:NS} with initial data $u_0$.
			It is easy to check the solution we find satisfies the Definition \ref{def:energy persol}.
			Next, we prove the pointwise estimates in Theorem \ref{thm:main}. We have by Theorem \ref{prop:weighted_H1est} that
			\begin{equation*}
				|u(x,t)| = \frac{1}{\sqrt{t}}\left|v\left(\frac{x}{\sqrt{t}}\right)\right|\le C(A,\beta)\frac{1}{\sqrt{t}}\frac{1}{1+\frac{|x|}{\sqrt{t}}} =\frac{C(A,\beta)}{|x|+\sqrt{t}},
			\end{equation*}
			and
			\begin{equation*}
				|u(x,t)-e^{t\Delta}u_0| = \frac{1}{\sqrt{t}}\left|v_{\textrm{re}}\left(\frac{x}{\sqrt{t}}\right)\right|\le C(A,\beta)\frac{1}{\sqrt{t}}\frac{1}{1+\frac{|x|^2}{t}} =\frac{C(A,\beta)\sqrt{t}}{|x|^2+t},
			\end{equation*}
			and
			\begin{equation*}
				\begin{aligned}
					|\nabla u(x,t)- \nabla e^{t\Delta}u_0| &= \frac{1}{t}\left|\nabla v_{\textrm{re}}\left(\frac{x}{\sqrt{t}}\right)\right|\le C(A,\beta)\frac{1}{t}\frac{1}{1+\frac{|x|^{2+\beta}}{t^{\frac{2+\beta}{2}}}}\\
					&\leq \frac{C(A,\beta)t^{\frac{\beta}{2}}}{(|x|+\sqrt{t})^{2+\beta}}.
				\end{aligned}	
			\end{equation*}
			At last,
			if the initial data $u_0\in C^2(S^1)$, we have the following optimal decay rate estimates
			\begin{equation*}
				\begin{aligned}
					|\nabla u(x,t)- \nabla e^{t\Delta}u_0| &= \frac{1}{\sqrt{t}}\left|\nabla v_{\textrm{re}}\left(\frac{x}{\sqrt{t}}\right)\right|\\
					&\leq C(\|u_0\|_{C^2(S^1)})\frac{t}{(|x|+\sqrt{t})^3}\ln\left(1+\frac{|x|}{\sqrt{t}}\right).
				\end{aligned}	
			\end{equation*}
            Hence the proof of the theorem is complete.
		\end{proof}

	\appendix
    \section{Proof of Lemma \ref{lem:lineardecay}}\label{app:lineardecay}
In this appendix, we prove Lemma \ref{lem:lineardecay}. Before we proceed, we give the following proposition concerning the
  estimates for $(-1)$-homogeneous, locally H\"older continuous vector field in any dimension, which is needed in the proof of Lemma \ref{lem:lineardecay}.
\begin{proposition}\label{prop:Holderdecay}
Let $u_0(x)$ be a $(-1)$-homogeneous vector field in $\mathbb{R}^n \setminus \{0\}$ ($n\geq 2$) such that its restriction to the unit sphere satisfies $u_0 \in C^{0,\beta}(S^{n-1})$ with $0 < \beta < 1$. Let $A = \|u_0\|_{C^{0,\beta}(S^{n-1})}$. Then for any $y, \tilde{y} \in \mathbb{R}^n \setminus \{0\}$, the following estimate holds:
\begin{equation} \label{eq:u0_holder_global}
    |u_0(y) - u_0(\tilde{y})| \leq C(A) \frac{|y-\tilde{y}|^\beta}{\min( |y|,|\tilde{y}|)^{1+\beta}}.
\end{equation}
\end{proposition}

\begin{proof}
Let $y, \tilde{y} \in \mathbb{R}^n \setminus \{0\}$. We use spherical coordinates $y = \rho \theta$ and $\tilde{y} = \tilde{\rho} \tilde{\theta}$, where $\rho = |y|$, $ \tilde{\rho} = |\tilde{y}|$, and $\theta, \tilde{\theta} \in S^{n-1}$. By the $(-1)$-homogeneity of $u_0$, we have $u_0(y) = \rho^{-1}u_0(\theta)$ and $u_0(\tilde{y} ) = \tilde{\rho} ^{-1}u_0(\tilde{\theta})$. The difference is decomposed and estimated as
\begin{equation*}
    \begin{aligned}
        | u_0(y)-u_0(\tilde{y}) | &= \left| \frac{1}{\rho} u_0(\theta) -\frac{1}{\tilde{\rho} } u_0(\tilde{\theta})   \right| \\
        &= \left| \frac{1}{ \tilde{\rho}} (  u_0(\theta)-u_0(\tilde{\theta})) + \left(   \frac{1}{\rho} -\frac{1}{ \tilde{\rho}} \right) u_0(\theta) \right| \\
        &\leq \frac{1}{ \tilde{\rho}} |  u_0(\theta)-u_0(\tilde{\theta})| + |u_0(\theta)| \left|   \frac{1}{\rho} -\frac{1}{ \tilde{\rho}}\right| =:\textrm{I} + \textrm{II}.
    \end{aligned}
\end{equation*}

\textbf{Step 1: Estimating Term I.}
Since $u_0 \in C^\beta(S^{n-1})$, we have $ u_0(\theta)-|u_0(\tilde{\theta}) | \leq A |  \theta-\tilde{\theta}|^\beta$. Utilizing the geometric inequality 
\begin{equation*}
    \left| \theta -\tilde{\theta} \right| = \left|  \frac{y}{|y|}-\frac{\tilde{y}}{|\tilde{y}|}\right| \leq \frac{2|y-\tilde{y}|}{\max( |y|,|\tilde{y}|)},
\end{equation*}
we obtain
\begin{equation*}
    \frac{1}{\tilde{\rho} } |u_0(\tilde{\theta}) - u_0(\theta)| \leq \frac{A}{ \tilde{\rho}} \left( \frac{2|y-\tilde{y}|}{\max( \tilde{\rho}, \rho)} \right)^\beta \leq \frac{C(A)}{\min( \tilde{\rho}, \rho)^{1+\beta}} |y-\tilde{y}|^\beta.
\end{equation*}

\textbf{Step 2: Estimating Term II.}
Without loss of generality, assume $\rho \leq \tilde{\rho} $. We prove the term II is bounded by $C \rho^{-(1+\beta)}|y-\tilde{y} |^\beta$ by considering two cases for $| \rho-\tilde{\rho}|$.

\textit{Case A: Near-field radial difference ($| \rho-\tilde{\rho}| \leq \frac{\rho}{2}$).}
In this case, $ \tilde{\rho} \geq \rho$ and $| \rho-\tilde{\rho}| \leq |y-\tilde{y} |$. We have 
\begin{equation*}
    \left| \frac{1}{ \tilde{\rho}} - \frac{1}{\rho} \right| = \frac{| \rho-\tilde{\rho}|}{ \tilde{\rho}\rho} = \frac{| \rho-\tilde{\rho}|^\beta \cdot | \rho-\tilde{\rho}|^{1-\beta}}{\tilde{\rho} \rho} \leq \frac{|y-\tilde{y}|^\beta \cdot (\rho/2)^{1-\beta}}{\rho^2} = \frac{1}{2^{1-\beta}} \frac{|y-\tilde{y}|^\beta}{\rho^{1+\beta}}.
\end{equation*}

\textit{Case B: Far-field radial difference ($|\tilde{\rho} -\rho| > \frac{\rho}{2}$).}
In this case, $|y-\tilde{y}| \geq |\tilde{\rho} -\rho| > \frac{\rho}{2}$, which implies $1 < \frac{2^\beta |y-\tilde{y}|^\beta}{\rho^\beta}$. Since $\tilde{\rho}  \geq \rho$, we have
\begin{equation*}
    \left| \frac{1}{\tilde{\rho} } - \frac{1}{\rho} \right| \leq \frac{1}{\tilde{\rho} } + \frac{1}{\rho} \leq \frac{2}{\rho} < \frac{2}{\rho} \cdot \frac{2^\beta |y-\tilde{y}|^\beta}{\rho^\beta} = \frac{2^{1+\beta} |y-\tilde{y}|^\beta}{\rho^{1+\beta}}.
\end{equation*}

Combining both cases, we have $\textrm{ II} \leq \frac{C A}{\min(\tilde{\rho} , \rho)^{1+\beta}} |y-\tilde{y}|^\beta$.
Summing the estimates for terms I and  II yields \eqref{eq:u0_holder_global}. 
\end{proof}
We can now prove Lemma \ref{lem:lineardecay}.
    \begin{proof}[Proof of Lemma \ref{lem:lineardecay}]
		The smoothness of $v_0$ follows directly from the classical regularity theory for the heat equation.
		We now proceed to prove the estimate \eqref{eq:scale-invbou-2}.
		Throughout the proof, $C(A)$ ($C(A,\beta)$) denotes a constant depending only on $A$ ($A$ and $\beta$), which may vary from line to line. 
        We denote $G(x) = \frac{1}{4\pi} e^{-|x|^2/4}$ as the heat kernel at $t=1$ on $\mathbb{R}^2$ in this proof.
        Then
		\begin{equation*}
			v_0(y) = \int_{\mathbb{R}^2} G(y-z)u_0( z)dz,
		\end{equation*}
        and by our assumption $u_0( z) = \frac{1}{|z|}u_0\left( \frac{z}{|z|}\right) $ with $\|u_0\|_{C^{0,\beta}(S^1)}\leq A$.
        
\textbf{The case $k=0$.}        
		For the region $|y| \leq 1$, $v_0(y)$ is clearly bounded. Specifically, when $|y|\leq 1$,
		\begin{equation}\label{eq:boundednessinside}
			\begin{aligned}
				| v_0(y) |&\leq  \left(\int_{|z|\leq 2} + \int_{|z|\geq 2}  \right) e^{-\frac{| y-  z|^2}{4  }}\frac{C(A)}{|z|}dz\\
				& \leq \left(\int_{|z|\leq 2} \frac{C(A)}{|z|}dz + \int_{|z|\geq 2}  e^{-\frac{| y-  z|^2}{4  }}\frac{C(A)}{2}dz\right)\\
				&\le C(A) +C(A)\int_{|z|\geq 2}  e^{-\frac{|  z|^2}{16 }}dz\le C(A).
			\end{aligned}           
		\end{equation}
		For $|y|>1$,
        \begin{equation*}
			\begin{aligned}
				| v_0(y) |&\leq  \left(\int_{|z|\leq \frac{|y|}{2}} + \int_{ \frac{|y|}{2}\leq |z| \leq 2|y|} + \int_{|z|\geq 2|y|}  \right) e^{-\frac{| y-  z|^2}{4  }}\frac{C(A)}{|z|}dz\\
                &:=I_1+I_2+I_3.
			\end{aligned}           
		\end{equation*}
        When $|z|\leq \frac{|y|}{2}$, we know $|y-z|\geq \frac{|y|}{2}$, and then
        \begin{equation*}
			\begin{aligned}
            I_1\leq e^{-\frac{|y|^2}{16}} \int_{|z|\leq \frac{|y|}{2}} \frac{C(A)}{|z|}dz \leq C(A)e^{-\frac{|y|^2}{16}} |y|.
            \end{aligned}           
		\end{equation*}
        For $I_2$, 
        we first note that $\left\{\frac{|y|}{2}\leq |z| \leq 2|y|\right\}\subset \left\{z: |y-z| \leq 3|y|\right\}$. Hence,
        \begin{equation*}
			\begin{aligned}
            I_2 &\leq \int_{\left\{z: |y-z| \leq 3|y|\right\}} e^{-\frac{| y-  z|^2}{4  }}\frac{C(A)}{|y|} dz
             \leq \frac{C(A)}{|y|} \int_{\left\{ |\tilde{z}| \leq 3|y|\right\}} e^{-\frac{| \tilde{z}|^2}{4  }} dz \leq \frac{C(A)}{|y|}. 
            \end{aligned}           
		\end{equation*}
        For $I_3$, we note that when $|z|\geq 2|y|$, we have $|y|\leq \frac{|z|}{2}$, $|y-z|\geq \frac{|z|}{2}$ and also $|y-z|\geq |y|$. Therefore, we 
      can  estimate  $I_3$ as
        \begin{equation*}
			\begin{aligned}
            I_3 &\leq  \frac{C(A)}{|y|}   \int_{|z|\geq 2|y|} e^{-\frac{|y-z|^2}{8}} e^{-\frac{|y-z|^2}{8}} dz
            \leq
            \frac{C(A)}{|y|}  e^{-\frac{|y|^2}{8}} \int_{|z|\geq 2|y|} e^{-\frac{|z|^2}{32}} dz \leq \frac{C(A)}{|y|}  e^{-\frac{|y|^2}{8}}.
            \end{aligned}           
		\end{equation*}
        This shows that for $|y|>1$, we have
        \begin{equation}\label{eq:boundednessoutside}
            \begin{aligned}
             |v_0(y)|\leq \frac{C(A)}{|y|}.   
            \end{aligned}
        \end{equation}
        Combining \eqref{eq:boundednessinside} and \eqref{eq:boundednessoutside}, we obtain 
        \begin{equation*}
            |v_0(y)|\leq \frac{C(A)}{1+|y|}.   
        \end{equation*}
    This completes the proof of the \eqref{eq:scale-invbou-2} for $k=0$.     
  
\textbf{The case $k=1$.}   			
			Next, we prove \eqref{eq:scale-invbou-2} for $k=1$. The case for $\beta =0$ is trivial.
            We then consider the case when the component $\beta=1$.  We  have that
			\begin{equation*}
				\partial_i  v_0(y) = \frac{1}{4\pi}\int_{\mathbb{R}^2} e^{-\frac{| y-  z|^2}{4  }}\left(-\frac{1}{2}\right)(y_i-z_i)u_0( z)dz, ~ i=1,2.
			\end{equation*}
			We have
			\begin{equation*}
				\begin{aligned}
					|	\partial_i  v_0(y) |&\leq \int_{\mathbb{R}^2} e^{-\frac{| y-  z|^2}{4  }}|y_i-z_i|\frac{A}{|z|}dz	\leq \int_{\mathbb{R}^2} e^{-\frac{| y-  z|^2}{4  }}|y-z|\frac{A}{|z|}dz.
				\end{aligned}
			\end{equation*}
			Hence, when $|y|\leq 1$,
			\begin{equation*}
				\begin{aligned}
					|\partial_i  v_0(y) |&\leq  \left(\int_{|z|\leq 2} + \int_{|z|\geq 2}  \right) e^{-\frac{| y-  z|^2}{4  }}|y-z| \frac{C(A)}{|z|}dz\\
					& \leq \left(\int_{|z|\leq 2} \frac{C(A)}{|z|}dz + \int_{|z|\geq 2}  e^{-\frac{| y-  z|^2}{4  }}|y-z|C(A)dz\right)\\
					&\le C(A) +C(A)\int_{|z|\geq 2}  e^{-\frac{|  z|^2}{16 }}dz\le C(A).
				\end{aligned}
			\end{equation*}
			On the other hand, when $|y|\geq 1$, we rewrite
			\begin{equation*}
				\begin{aligned}
					\partial_i  v_0(y) &= \frac{1}{4\pi}\int_{|z|\leq \frac{1}{2}} + \int_{|z|\geq \frac{1}{2}}  e^{-\frac{| y-  z|^2}{4  }}\left(-\frac{1}{2}\right)(y_i-z_i)u_0( z)dz\\
					&= \frac{1}{4\pi}\int_{|z|\leq \frac{1}{2}}  e^{-\frac{| y-  z|^2}{4  }}\left(-\frac{1}{2}\right)(y_i-z_i)u_0( z)dz  +  \frac{1}{4\pi}\int_{|z|\geq \frac{1}{2}} e^{-\frac{| y-  z|^2}{4  }} \partial_i u_0( z)dz \\
					&\quad -\frac{1}{4\pi}\int_{\partial B_\frac{1}{2}} e^{-\frac{| y-  z|^2}{4  }}  u_0( z)\cdot n  d\sigma_z =J_1 + J_2 + J_3.
				\end{aligned}
			\end{equation*}
			Similar to estimate of $ v_0$ above, it is easy to find that for $|y|>1$.
			\begin{equation*}
				|J_1|\leq \int_{|z|\leq \frac{1}{2}}  e^{-\frac{| y-  z|^2}{4  }}|y-z|\frac{C(A)}{|z|}dz \leq  \int_{|z|\leq \frac{1}{2}}  e^{-\frac{| y|^2}{16}}\frac{|y|}{2}\frac{C(A)}{|z|}dz \leq \frac{C(A)}{1+|y|^2}.
			\end{equation*}
			\begin{equation*}
				|J_2|\leq \int_{|z|\geq \frac{1}{2}} e^{-\frac{| y-  z|^2}{4  }} \frac{C(A)}{|z|^2} dz \leq \frac{C(A)}{1+|y|^2}.
			\end{equation*}
			\begin{equation*}
				|J_3|\le C(A)\int_{\partial B_\frac{1}{2}} e^{-\frac{| y-  z|^2}{4  }}  d\sigma_z\leq \frac{C(A)}{1+|y|^2}.
			\end{equation*}
			We finish the proof of \eqref{eq:scale-invbou-2} for $k=1$ and $\beta=1$. 
			
			Next, we consider the case for $k=1$ and  $0 < \beta < 1$. The derivative $\partial_i v_0(y)$ can be expressed as
			\begin{equation*}
				\partial_i v_0(y) = \int_{\mathbb{R}^2} \partial_{y_i} G(y-z) u_0(z) \, dz,
			\end{equation*}
			where $G(x) = \frac{1}{4\pi} e^{-|x|^2/4}$ is the heat kernel at $t=1$. Since $\int_{\mathbb{R}^2} \partial_{y_i} G(y-z) \, dz = 0$, we can utilize the following identity to handle the singularity and the decay:
			\begin{equation} \label{eq:cancel_id}
				\partial_i v_0(y) = \int_{\mathbb{R}^2} \partial_{y_i} G(y-z) [u_0(z) - u_0(y)] \, dz.
			\end{equation}
            The proof of boundedness of $\partial_i v_0$ when $|y|\leq 1$ is similar to the above computations.
For $|y| > 1$, we divide the integration domain into three regions similar as above: $L_1 = \{ z: |z| \leq \frac{|y|}{2} \}$ and $L_2 = \{ z:  \frac{|y|}{2}\leq |z|\leq 2|y| \}$ and $L_3 = \{ z : |z| \geq 2|y| \}$.
			
	\textbf{Case 1: Region $L_1$.} In  region $L_1$, we revert to the original form of the integral using the identity \eqref{eq:cancel_id}:
			\begin{equation*}
				\int_{L_1} \partial_{y_i} G(y-z) u_0(z) \, dz - u_0(y) \int_{L_1} \partial_{y_i} G(y-z) \, dz =: M_1 + M_2.
			\end{equation*}
            For $M_1$, we have
            \begin{equation*}
                \begin{aligned}
                 |M_1|    \leq C(A) e^{-\frac{|y|^2}{16}} \int_{|z| \leq |y|/2} \frac{1}{|z|}\, dz \leq C(A) e^{-\frac{|y|^2}{16}} |y|.
                \end{aligned}
            \end{equation*}
			For $M_2$, since $|u_0(y)| \leq C|y|^{-1}$, 
            we have
			\begin{equation*}
				|M_2| \leq C(A) e^{-\frac{|y|^2}{16}}\frac{1}{|y|} \int_{|z| \leq |y|/2} \, dz \leq C(A) e^{-\frac{|y|^2}{16}} |y|.
			\end{equation*}
            \textbf{Case 2: Region $L_2$.} 
            It follows from Proposition \ref{prop:Holderdecay} in region $L_2$ that
			\begin{equation*}
				|u_0(z) - u_0(y)| \le C(A) \frac{|z-y|^\beta}{\min(|z|, |y|)^{1+\beta}} \leq \frac{C(A)}{|y|^{1+\beta}} |z-y|^\beta.
			\end{equation*}
            In this region, for $z \in L_2$, we have $|y-z|\leq  3|y|$. 
			We get  
			\begin{equation*}
				\begin{aligned}
					\left| \int_{L_2} \partial_{y_i} G(y-z) [u_0(z) - u_0(y)] \, dz \right| 
					&\leq \frac{C(A)}{|y|^{1+\beta}} \int_{|\tilde{z}| \leq 3|y|} |\partial_i G(\tilde{z}) ||\tilde{z}|^{\beta}  \, d\tilde{z} 
					\leq \frac{C(A,\beta)}{|y|^{1+\beta}}.
				\end{aligned}
			\end{equation*}
\textbf{Case 3: Region $L_3$.} The estimates in the region are similar to Region $L_1$. We conclude that
\begin{equation*}
				\left|\int_{L_3} \partial_{y_i} G(y-z) u_0(z) \, dz - u_0(y) \int_{L_3} \partial_{y_i} G(y-z) \, dz\right| \leq C(A) e^{-\frac{|y|^2}{16}}.
			\end{equation*}	
			Combining the estimates for $L_1$, $L_2$ and $L_3$, we obtain that for $\beta \in (0, 1)$,
			\begin{equation*}
				|\partial_i v_0(y)| \leq \frac{C(A,\beta)}{1+|y|^{1+\beta}}.
			\end{equation*}
			The proofs for $k \geq 2$ follow by applying similar  arguments
            to the higher-order derivatives of the heat kernel, and we finish the proof. 
\end{proof}

		{\bf Acknowledgement.}
		The research of Gui is supported by  NSFC Key Program (Grant No. 12531010),  University of Macau research grants CPG2024-00016-FST, CPG2025-00032-FST, CPG2026-00027-FST, SRG2023-00011-FST, MYRG-GRG2023-00139-FST-UMDF, UMDF Professorial Fellowship of Mathematics, Macao SAR FDCT 0003/2023/RIA1 and  Macao SAR FDCT 0024/2023/RIB1.   The research of  Xie is partially supported by  NSFC grants 12571238 and 12426203.

	\end{document}